\documentclass[12pt]{amsart}
\usepackage[top=1in, bottom=1in, margin=1in]{geometry}
\usepackage{amsfonts,amsmath,amssymb,amscd,amsthm}
\usepackage[all]{xy}
\usepackage{enumerate}

\newcommand{\G}{\mathcal{G}}
\newcommand{\GU}{\G_{\mathcal{U}}}

\newtheorem{theorem}{Theorem}[section]
\newtheorem{prop}[theorem]{Proposition}
\newtheorem{lem}[theorem]{Lemma}
\newtheorem{cor}[theorem]{Corollary}
\newtheorem{rem}[theorem]{Remark}
\newtheorem{ass}[theorem]{Assumption}
\newtheorem{mydef}[theorem]{Definition}
\newtheorem{example}[theorem]{Example}

\begin{document}

\title{Weak Pullbacks of Topological Groupoids}
\author{A. Censor}
\address{Aviv Censor, School of Mathematical Sciences,
Tel-Aviv University, Tel-Aviv 69978, Israel}
\email{avivc@post.tau.ac.il}
\author{D. Grandini}
\address{Daniele Grandini, Department of Mathematics,
University of California at Riverside, Riverside, CA 92521, U.S.A.}
\email{daniele@math.ucr.edu}
\date{\today}
\begin{abstract}
We introduce the category $\mathcal{HG}$, whose objects are topological groupoids endowed with compatible measure theoretic data: a Haar system and a measure on the unit space. We then define and study the notion of weak-pullback in the category of topological groupoids, and subsequently in $\mathcal{HG}$. The category $\mathcal{HG}$ is the setting for topological groupoidification, which we present in separate papers, and in which the weak pullback is a key ingredient.
\end{abstract}
%Integration and disintegration of measures

\maketitle

%***************************************************************
%***************************************************************
\section{Introduction}
%***************************************************************
%***************************************************************

The leading actors in this paper are groupoids that we call \textit{Haar groupoids\footnote{A discussion regarding terminology appears at the end of this introduction.}}. A Haar groupoid is a topological groupoid endowed with certain compatible measure theoretic ingredients. More precisely, a Haar groupoid is a locally compact, second countable, Hausdorff groupoid $G$, which admits a continuous left Haar system $\lambda^{\bullet}$, and is equipped with a non-zero Radon measure $\mu^{(0)}$ on its unit space $G^{(0)}$, such that $\mu^{(0)}$ is quasi-invariant with respect to $\lambda^{\bullet}$. Maps between Haar groupoids are continuous groupoid homomorphisms, which respect the extra structure in an appropriate sense. One is naturally led to define a category, which we denote by $\mathcal{HG}$, the category of Haar groupoids. Section \ref{sec:preliminaries} introduces this category.

A general study of the category $\mathcal{HG}$ from a purely categorical perspective will be presented in a separate paper. In this paper we focus on one specific categorical notion, namely the \textit{weak pullback}. We first construct the weak pullback of topological groupoids. The weak pullback of the following given cospan diagram of topological groupoids and continuous homomorphisms:
\[
\xymatrix{
& S \ar[dr]_{p} & & T \ar[dl]^{q} & \\
 & & G & &
}
\]
is a topological groupoid $P$ along with projections $\pi_S : P \rightarrow S$ and $\pi_T : P \rightarrow T$, which together give rise to the following diagram (which does \textit{not} commute):
\[
\xymatrix{
& & P \ar[dl]_{\pi_S} \ar[dr]^{\pi_T} & &\\
& S \ar[dr]_{p} & & T \ar[dl]^{q} & \\
 & & G & &
}
\]
As a set, $P$ is contained in the cartesian product $S \times G \times T$, from which it inherits its topology. The elements of $P$ are triples of the form $(s,g,t)$, where $p(s)$ and $q(t)$ are not equal to $g$, but rather in the same orbit of $G$ via $g$. More precisely, denoting the range and source maps of $G$ by $r_G$ and $d_G$ respectively, $$P := \{ (s,g,t) ~|~ s\! \in\! S, \ g\! \in\! G, \ t \!\in\! T, \; r_G(g) \!=\! r_G(p(s)) \ \text{ and } \ d_G(g) \!=\! r_G(q(t)) \}.$$ The groupoid structure of $P$ is described in Section \ref{sec:topological wpb}, followed by a discussion of its properties. In the discrete groupoid setting, our notion of weak pullback reduces to the one introduced by Baez et al. in \cite{baez-hoffnung-walker}, which in turn generalizes the more familiar notion of a pullback in the category of sets.

Upgrading the weak pullback from topological groupoids to the category $\mathcal{HG}$ requires non-trivial measure theory and analysis. In Section \ref{sec:Haar system for wpb} we construct a Haar system for $P$. Section \ref{sec:measure for wpb} is then devoted to creating a quasi invariant measure on $P^{(0)}$. Finally, in Section \ref{sec:WPB for HG}, we prove that with these additional ingredients, subject to a certain additional assumption, we indeed obtain a weak pullback in $\mathcal{HG}$.

This paper is part of a project we are currently working on, in which we are extending groupoidification from the discrete setting to the realm of topology and measure theory. \emph{Groupoidification} is a form of categorification, introduced by John Baez and James Dolan. It has been successfully applied to several structures, which include Feynman Diagrams, Hecke Algebras and Hall Algebras. An excellent account of groupoidification and its triumphs to date can be found in \cite{baez-hoffnung-walker}. So far, the scope of groupoidification and its inverse process of degroupoidification has been limited to purely algebraic structures and discrete groupoids. The category $\mathcal{HG}$ provides the setting for our attempt at topological groupoidification, in which the notion of the weak pullback plays a vital role. This line of research is pursued in separate papers.

This paper relies heavily on general topological and measure theoretic techniques related to Borel and continuous systems of measures and their mapping properties. A detailed study of this necessary background theory appears in our paper \cite{BSM}, from which we quote many definitions and results and to which we make frequent references throughout this text.

%***************************************************************
%***************************************************************
\subsection{A note about terminology}\label{subsec:Haar groupoid terminology}
%***************************************************************
%***************************************************************

Seeking a distinctive name for the groupoids we consider in these notes and in our subsequent work on topological groupoidification, we opted to call them ``Haar groupoids". These groupoids bear close resemblance to \emph{measure groupoids} with Haar measures, as studied by Peter Hahn in \cite{Hahn1}, following Mackey \cite{Mackey} and Ramsay \cite{Ramsay virtual groups}, leading to the theory of groupoid von-Neumann algebras. Like the groupoids we consider, measure groupoids carry a measure (or measure class), which admits a disintegration via the range map, namely what is nowadays known as a Haar system. The main discrepancies are that we require our groupoids to exhibit a nice topology (locally compact, Hausdorff) and to be endowed with a \emph{continuous} Haar system, whereas measure groupoids need only have a Borel structure in general, and host Borel Haar systems. 

Locally compact topological groupoids which may admit continuous Haar systems are as well studied in the literature as measure groupoids, in particular as part of groupoid $C^*$-algebra theory as developed by Jean Renault in \cite{renault-book} (other standard references include \cite{muhly-book-unpublished} and \cite{paterson-book}). In many cases locally compact groupoids indeed exhibit the full structure of our Haar groupoids, yet the literature does not single them out terminology-wise.

%***************************************************************
%***************************************************************
\section{Preliminaries and the category $\mathcal{HG}$}\label{sec:preliminaries}
%***************************************************************
%***************************************************************

We begin by fixing notation. We shall denote the \textit{unit space} of a groupoid $G$ by $G^{(0)}$ and
the set of \textit{composable pairs} by $G^{(2)}$. The \textit{range} (or target) and \textit{domain} (or source) maps of $G$ are denoted respectively by $r$ and $d$, or by $r_G$ and $d_G$ when disambiguation is necessary. We set $G^u=\{ x\in G ~|~ r(x)=u\}$, $G_v=\{ x\in G ~|~ d(x)=v\}$ and $G^u_v = G^u \cap G_v$, for all $u,v \in G^{(0)}$. Thus $G^u_u$ is the \textit{isotropy group} at $u$.

We let $\underline{G} = G^{(0)} / G = \{ [u] ~|~ u \in G^{(0)} \}$ denote the \textit{orbit space} of a groupoid $G$. The orbit space $\underline{G}$ inherits a topology from $G$ via $G^{(0)}$, defined by declaring $W \subseteq \underline{G}$ to be open whenever $q^{-1}(W)$ is open in $G^{(0)}$, where $q:G^{(0)} \longrightarrow \underline{G}$ is the quotient map $u \mapsto [u]$. \\

\emph{Throughout this paper, we will assume our topological groupoids to be second countable, locally compact and Hausdorff.} Any such groupoid $G$ is metrizable
and normal, and satisfies that every locally finite measure is $\sigma$-finite. Moreover, $G$ is a Polish space and hence strongly Radon, i.e. every locally finite Borel measure is a Radon measure. For more on Polish groupoids, we refer the reader to a paper by Ramsay \cite{Ramsay polish groupoids}. In general, however, $\underline{G}$ does \emph{not} necessarily inherit these properties, a fact that will require occasional extra caution. \\

Haar systems for groupoids play a key role in this paper. In the groupoid literature, modulo minor discrepancies between various sources (see for example standard references such as \cite{muhly-book-unpublished}, \cite{paterson-book}, \cite{renault-book} and \cite{renault-anantharaman-delaroche}), a continuous left Haar system is usually defined to be a family $\lambda= \{\lambda^u : u \in G^{(0)} \}$ of positive (Radon) measures on $G$ satisfying the following properties:
\begin{enumerate}
\item
$supp(\lambda^u) = G^u$ for every $u \in G^{(0)}$;
\item
for any $f \in C_c(G)$, the function $u \mapsto \int f d\lambda^u$ on $G^{(0)}$ is in $C_c(G^{(0)})$;
\item
for any $x \in G$ and $f \in C_c(G)$, $\int f(xy)d\lambda^{d(x)}(y) = \int f(y)d\lambda^{r(x)}(y).$
\end{enumerate}
In this paper we shall use Definition \ref{def:Haar system} below as our definition of a Haar system. It is taken from \cite{BSM}, where it is shown to be equivalent to the more common definition above. For the convenience of the reader we include here a very brief summary of the notions from \cite{BSM} that lead to Definition \ref{def:Haar system}, all of which we will use extensively throughout this paper. Henceforth, as in \cite{BSM}, all topological spaces are assumed to be second countable and $\mathbf{T_1}$ in general, and also locally compact and Hausdorff whenever dealing with continuous systems of measures.

Let $\pi:X \rightarrow Y$ be a Borel map. A \textit{system of measures} (\cite{BSM}, Definition 2.2) on $\pi$ is a family of (positive, Borel) measures $\lambda^{\bullet} = \{ \lambda^y \}_{y \in Y}$ such that:
\begin{enumerate}
\item Each $\lambda^y$ is a Borel measure on $X$;
\item For every $y$, $\lambda^y$ is concentrated on $\pi^{-1}(y)$.
\end{enumerate}
We will denote a map $\pi:X \rightarrow Y$ admitting a system of measures $\lambda^{\bullet}$ by the diagram
$\xymatrix{X\ar [rr]^{\pi}_{\lambda^{\bullet}}&&Y}$.

We will say that a system of measures $\lambda^{\bullet}$ is \textit{positive on open sets} (\cite{BSM}, Definition 2.3) if $\lambda^y(A) > 0$ for every $y \in Y$ and for every open set $A \subseteq X$ such that $A \cap \pi^{-1}(y) \neq \emptyset$.
A system of measures $\lambda^{\bullet}$ on a continuous map $\pi:X \rightarrow Y$ will be called
a \textit{continuous system of measures} or \textit{CSM} (\cite{BSM}, Definition 2.5) if for every non-negative continuous compactly supported function $0 \leq f\in C_c(X)$, the map $ y \mapsto \int_X f(x)d \lambda^y (x)$ is a continuous function on $Y$. A system of measures $\lambda^{\bullet}$ on a Borel map $\pi:X \rightarrow Y$ is called a \textit{Borel system of measures} or \textit{BSM} (\cite{BSM}, Definition 2.6) if for every Borel subset $E\subseteq X$, the function $\lambda^{\bullet}(E):Y\rightarrow [0,\infty]$ given by $y\mapsto \lambda^{y}(E)$ is a Borel function. A system of measures $\lambda^{\bullet}$ satisfying that every $x \in X$ has a neighborhood $U_x$ such that $\lambda^{y}(U_x)<\infty$ for every $y \in Y$, will be called \textit{locally finite} (\cite{BSM}, Definition 2.14), and \emph{locally bounded} if there is a constant $C_x >0$ such that $\lambda^y(U_x) < C_x$ for any $y \in Y$ (\cite{BSM}, Definition 2.3). A detailed discussion of the mutual relations between the above concepts appears in \cite{BSM}.

Let $G$ be a topological groupoid. A system of measures $\lambda^{\bullet}$ on the range map $r:G \rightarrow G^{(0)}$ is said to be a \textit{system of measures on $G$} (\cite{BSM}, Definition 7.1).
It is called \textit{left invariant} (\cite{BSM}, Definition 7.2) if for every $x \in G$ and for every Borel subset $E \subseteq G$, $$\lambda^{d(x)} (E) = \lambda^{r(x)}\left(x \cdot (E \cap G^{d(x)})\right).$$

\begin{mydef}\label{def:Haar system}(\cite{BSM}, Definition 7.5)
A continuous left \textbf{Haar system} for $G$ is a system of measures $\lambda^{\bullet}$ on $G$ which is continuous, left invariant and positive on open sets.
\end{mydef}

Playing side by side to the Haar system $\lambda^{\bullet}$, another leading actor in our work is a Radon measure on the unit space $G^{(0)}$ of a groupoid $G$, which we denote by $\mu^{(0)}$. The measure $\mu^{(0)}$ will be related to $\lambda^{\bullet}$ via the notion of quasi invariance, which we spell out below. We usually follow \cite{muhly-book-unpublished}, where the reader can find much more about the important role of quasi invariant measures in groupoid theory.

\begin{mydef}\label{def:induced measure muG}
Let $G$ be a groupoid admitting a Haar system $\lambda^{\bullet}$ and a Radon measure $\mu^{(0)}$ on $G^{(0)}$. The \textbf{induced measure} $\mu$ on $G$ is defined for any Borel set $E \subseteq G$ by the formula:
$$\mu (E) = \int_{G^{(0)}} \lambda^u (E) d\mu^{(0)} (u) .$$
\end{mydef}

\begin{lem}\label{lem:properties of induced measure}
The induced measure $\mu$ is a Radon measure on $G$.
\end{lem}

\begin{proof}
Since $G$ is strongly Radon, it suffices to prove that $\mu$ is locally finite. The induced measure $\mu$ is obtained as a composition of the system $\lambda^{\bullet}$ with the measure $\mu^{(0)}$. The Haar system $\lambda^{\bullet}$ is a CSM, hence a locally bounded BSM, by Lemma 2.11 and Proposition 2.23 of \cite{BSM}. In addition, the measure $\mu^{(0)}$ is locally finite. Therefore, the conditions of Corollary 3.7 in \cite{BSM} are met, and we conclude that $\mu$ is locally finite.
\end{proof}

The following simple observation will be useful in the sequel.

\begin{lem}\label{lem:integrating muG vs muG0}
For any Borel function $f$ on $G$:
$$\int_G f(x) d\mu(x) = \int_{G^{(0)}} \left( \int_G f(x) d\lambda^u (x) \right) d\mu^{(0)} (u).$$
\end{lem}

\begin{proof}
For every Borel subset $E \subseteq G$, by Definition \ref{def:induced measure muG}, $$\int_G \chi_{_E} (x) d\mu(x) = \mu (E) = \int_{G^{(0)}} \lambda^u (E) d\mu^{(0)} (u) = \int_{G^{(0)}} \left( \int_G \chi_{_E} (x) d\lambda^u (x) \right) d\mu^{(0)} (u).$$ Generalizing from $\chi_{_E}$ to any Borel function $f$ is routine.
\end{proof}

The image of $\mu$ under inversion is defined by $$\mu^{-1}(E) := \mu(E^{-1}) = \mu( \{ x^{-1} ~|~ x \in E \} ) $$
for any Borel set $E \subseteq G$.

\begin{rem}\label{rem:relating muG to its inverse}
It is a standard exercise to show that for any Borel function $f$,
$$\int_G f(x) d\mu^{-1}(x) = \int_G f(x) d\mu (x^{-1}).$$
\end{rem}

\begin{mydef}\label{def:quasi-invariant measure}
Let $G$ be a groupoid admitting a Haar system $\lambda^{\bullet}$ and a Radon measure $\mu^{(0)}$ on $G^{(0)}$. The measure $\mu^{(0)}$ is called \textbf{quasi invariant} if the induced measure $\mu$ satisfies $\mu \sim \mu^{-1}$.
\end{mydef}
\noindent Here $\sim$ denotes equivalence of measures in the sense of being mutually absolutely continuous.

\begin{rem}\label{rem:Delta}
Let $\mu^{(0)}$ be quasi invariant. The Radon-Nikodym derivative $\Delta = d\mu / d\mu^{-1}$ is called the modular function of $\mu$. Although $\Delta$ is determined only a.e., it can be chosen (\cite{muhly-book-unpublished}, Theorem 3.15) to be a homomorphism from $G$ to $\mathbb{R}^{\times}_+$, so we will assume this to be the case. Recall that for any Borel function $f$,
\begin{equation}\label{eq:Delta}
\int_G f(x) d\mu(x) = \int_G f(x) \Delta (x) d\mu^{-1} (x).
\end{equation}
Furthermore, $\Delta^{-1} = d\mu^{-1} / d\mu$ satisfies the useful formula
\begin{equation}\label{eq:Delta inverse}
\int_G f(x) \Delta^{-1} (x) d\mu(x) = \int_G f(x^{-1}) d\mu(x),
\end{equation}
since $\int_G f(x) \Delta^{-1} (x) d\mu(x) = \int_G f(x) d\mu^{-1} (x) = \int_G f(x) d\mu (x^{-1}) = \int_G f(x^{-1}) d\mu(x)$ by Remark \ref{rem:relating muG to its inverse}.
\end{rem}

\begin{mydef}\label{def:Haar groupoid}
Let $G$ be a topological groupoid, which satisfies the following assumptions:
\begin{enumerate}
\item
The topology of $G$ is locally compact, second countable and Hausdorff.
\item
$G$ admits a continuous left Haar system $\lambda^{\bullet}$.
\item
$G^{(0)}$ is equipped with a non-zero Radon measure $\mu^{(0)}$ which is quasi-invariant with respect to $\lambda^{\bullet}$.
\end{enumerate}
Such a groupoid will be called a \textbf{Haar groupoid}.
\end{mydef}

We will denote a Haar groupoid by $(G,\lambda^{\bullet},\mu^{(0)})$, or just by $G$ when $\lambda^{\bullet}$ and $\mu^{(0)}$ are evident from the context.

\begin{mydef}\label{def:Haar groupoid homomorphism}
Let $(G,\lambda^{\bullet},\mu^{(0)})$ and $(H,\eta^{\bullet},\nu^{(0)})$ be Haar groupoids. Let $p:G\rightarrow H$ be a continuous groupoid homomorphism which is also measure class preserving with respect to the induced measures, i.e. $p_*(\mu) \sim \nu$. We say that $p$ is a \textbf{homomorphism of Haar groupoids}.
\end{mydef}
\noindent In the above definition $p_*$ is the push-forward, defined for any Borel set $E \subset H$ by $p_* \mu(E) = \mu (p^{-1}(E))$. A homomorphism of Haar groupoids is also measure class preserving on the unit spaces, as we shall shortly see. We first need the following fact.
\begin{lem}\label{lem:rG is measure class preserving}
Let  $(G,\lambda^{\bullet},\mu^{(0)})$ be a Haar groupoid. The range map $r:G \rightarrow G^{(0)}$ satisfies $r_*(\mu) \sim \mu^{(0)}$.
\end{lem}

\begin{proof}
Let $E \subseteq G^{(0)}$ be a Borel subset. We need to show that $\mu(r^{-1}(E)) =0$ if and only if $\mu^{(0)}(E) \!=\! 0$. By the definition of the induced measure, $\mu (r^{-1}(E)) \!=\! \int_{G^{(0)}} \lambda^u (r^{-1}(E)) d\mu^{(0)} (u)  \!=\! \int_{G^{(0)}} \chi_{_E}(u) \lambda^u (G) d\mu^{(0)} (u)$, since $\lambda^u (r^{-1}(E))  \!=\! 0$ if $u \notin E$ whereas $\lambda^u (r^{-1}(E)) = \lambda^u (G)$ if $u\in E$.
Since $\lambda^{\bullet}$ is a Haar system, $supp(\lambda^u) = G^u \neq \emptyset$, and in particular $\lambda^u (G) >0$ for every $u$. It follows that $\mu(r^{-1}(E)) =0$ if and only if $\chi_{_E}(u) = 0$  $\mu^{(0)}$-a.e., which is if and only if $\mu^{(0)}(E) = 0$.
\end{proof}
\noindent While the proof we included above is elementary, we point out that Lemma \ref{lem:rG is measure class preserving} also follows from the fact that by the definition of the induced measure $\mu$, the Haar system $\lambda^{\bullet}$ is a disintegration of $\mu$ with respect to $\mu^{(0)}$, which implies that $r:G \rightarrow G^{(0)}$ is measure class preserving. See Lemma 6.4 of \cite{BSM}.

Slightly abusing notation, we also denote the restriction of $p$ to $G^{(0)}$ by $p$.
\begin{prop}
Let $(G,\lambda^{\bullet},\mu^{(0)})$ and $(H,\eta^{\bullet},\nu^{(0)})$ be Haar groupoids, and let $p:G\rightarrow H$ be a homomorphism of Haar groupoids. Then $p_*(\mu^{(0)}) \sim \nu^{(0)}$.
\end{prop}

\begin{proof}
Consider the following commuting diagram:
\[
\xymatrix{G \ar [dd]_{p} \ar [rr]^{r_G} && G^{(0)} \ar [dd]_{p}\\\\
H \ar [rr]^{r_H} && H^{(0)} }
\]
Let $E \subseteq H^{(0)}$ be a Borel subset. We need to show that $\mu^{(0)}(p^{-1}(E))=0$ if and only if $\nu^{(0)}(E)=0$. Indeed, by Lemma \ref{lem:rG is measure class preserving} applied to $H$, $\nu^{(0)}(E) = 0 \Leftrightarrow \nu(r_H^{-1}(E)) = 0 \Leftrightarrow \mu(p^{-1}(r_H^{-1}(E))) = 0$. At the same time, by Lemma \ref{lem:rG is measure class preserving} applied to $G$, we have that $\mu^{(0)}(p^{-1}(E)) = 0 \Leftrightarrow \mu(r_G^{-1}(p^{-1}(E))) = 0$. Since the diagram commutes, $p^{-1}(r_H^{-1}(E)) = r_G^{-1}(p^{-1}(E))$, and it follows that $\nu^{(0)}(E) = 0 \Leftrightarrow \mu^{(0)}(p^{-1}(E)) = 0$.
\end{proof}

Having defined Haar groupoids and their appropriate maps, we are ready to define the setting for this paper and its sequels.
\begin{mydef}\label{def:category HG}
We introduce the category $\mathcal{HG}$, which has Haar groupoids as objects and homomorphisms of Haar groupoids as morphisms.
\end{mydef}

%***************************************************************
%***************************************************************
\section{The topological weak pullback}\label{sec:topological wpb}
%***************************************************************
%***************************************************************

The purpose of this paper is to construct and study the weak pullback of Haar groupoids. We start by constructing the weak pullback of \textit{topological} groupoids. We shall leave it to the reader to verify that in the case of \textit{discrete} groupoids, our notion of weak pullback reduces to the one in \cite{baez-hoffnung-walker}, which in turn generalizes the more familiar notion of pullback in the category of sets. Examples \ref{ex:open cover WPB} and \ref{ex:transformation groupoid WPB} below illustrate that the weak pullback is a natural notion.

\begin{mydef}\label{def:weak pullback}
Given the following diagram of topological groupoids and continuous homomorphisms
\[
\xymatrix{
& S \ar[dr]_{p} & & T \ar[dl]^{q} & \\
 & & G & &
}
\]
we define the \textbf{weak pullback} to be the topological groupoid
$$P = \{ (s,g,t) ~|~ s\! \in\! S, \ g\! \in\! G, \ t \!\in\! T, \; r_G(g) \!=\! r_G(p(s)) \ \text{ and } \ d_G(g) \!=\! r_G(q(t)) \}$$ together with the obvious projections $\pi_S : P \rightarrow S$ and $\pi_T : P \rightarrow T$. We describe the groupoid structure of $P$ and its topology below.
\end{mydef}

The weak pullback groupoid $P$ gives rise to the following diagram:
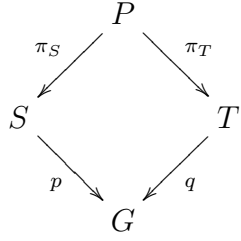
\[
\xymatrix{
& & P \ar[dl]_{\pi_S} \ar[dr]^{\pi_T} & &\\
& S \ar[dr]_{p} & & T \ar[dl]^{q} & \\
 & & G & &
}
\]
Observe that even at the level of sets, this diagram does \emph{not} commute. However, it is not hard to see that the weak pullback does make the following diamond commute:
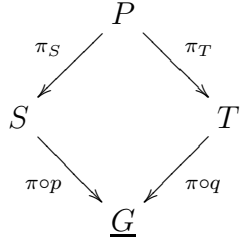
\[
\xymatrix{
& & P \ar[dl]_{\pi_S} \ar[dr]^{\pi_T} & &\\
& S \ar[dr]_{\pi \circ p} & & T \ar[dl]^{\pi \circ q} & \\
 & & \underline{G} & &
}
\]
where $\pi:G \longrightarrow \underline{G}$ is the map $g \longmapsto [r(g)] = [d(g)]$. \\

Intuitively, we think of an element $(s,g,t)$ in $P$ as giving rise to the following picture in $G$:

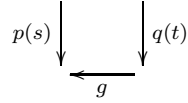
\[
\xymatrix{
&  \ar[d]_{p(s)} & \ar[d]^{q(t)} & \\
 & & \ar[l]^{g} & &
}
\]

Composition of $(s,g,t)$ and $(\sigma,h,\tau)$ is then thought of as:

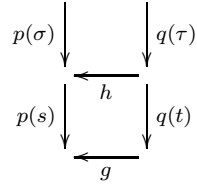
\[
\xymatrix{
&  \ar[d]_{p(\sigma)} & \ar[d]^{q(\tau)} & \\
&  \ar[d]_{p(s)} & \ar[l]^{h} \ar[d]^{q(t)} & \\
 & & \ar[l]^{g} & &
}
\]

Formally, the composable pairs of $P$ are $$ P^{(2)} = \{ (s,g,t),(\sigma,h,\tau) ~|~ r_S(\sigma)\!=\!d_S(s), r_T(\tau)\!=\!d_T(t) \text{ and }h\!=\!p(s)^{-1} gq(t) \} .$$ The product is given by $$(s,g,t)(\sigma,h,\tau) = (s \sigma ,g,t \tau),$$ and the inverse is given by $$  (s,g,t)^{-1} = (s^{-1},p(s)^{-1} gq(t),t^{-1} ).$$ Thus the range and source maps of $P$ are $$r_P(s,g,t) = (r_S(s), g, r_T(t))$$ and $$d_P(s,g,t) = (d_S(s), p(s)^{-1} gq(t), d_T(t)).$$ The unit space of $P$ is $$P^{(0)} = \{ (s,g,t) ~|~ s\in S^{(0)}, t \in T^{(0)} \text{ and } g \in G^{p(s)}_{q(t)}  \}.$$ The topology of $P$ is induced from the Cartesian product $S \times G \times T$: $$X \subseteq P \text{ is open } \Leftrightarrow \text{ there exists an open set } Z \subseteq S \times G \times T \text{ such that } X = Z \cap P.$$ The product and inverse of $P$ are continuous with respect to this topology.

\begin{rem}\label{rem:elementary open subsets of P}
Let $\{A_n\}_{n=1}^{\infty}$, $\{B_m\}_{m=1}^{\infty}$ and $\{C_k\}_{k=1}^{\infty}$ be countable bases for the topologies of $S$, $G$ and $T$ respectively. Then ${\mathcal B}=\{(A_n \times B_m \times C_k)\cap P \}_{n,m,k=1}^{\infty}$ gives a countable basis ${\mathcal B}$ for the topology of $P$, consisting of open sets of the form $E= (A \times B \times C) \cap P$, which we call \emph{elementary open sets}. Moreover, all finite intersections of sets in ${\mathcal B}$ are also of the this form.
\end{rem}

\begin{lem}\label{lem:P is LC+T2+2nd}
The groupoid $P$ is locally compact, Hausdorff and second countable.
\end{lem}

\begin{proof}
The groupoid $P$ is second countable by Remark \ref{rem:elementary open subsets of P}, and it is Hausdorff as a subspace of $S \times G \times T$. Let $$b: S\times G \times T\longrightarrow G^{(0)}\times G^{(0)}\times G^{(0)}\times G^{(0)}$$ be the continuous map given by $$(\sigma, x, \tau) \longmapsto (r_G(p(\sigma)), r_G(x), d_G(x), r_G(q(\tau))).$$ Observe that $P=b^{-1}(\Delta\times\Delta)$, where $\Delta$ is the diagonal of $G^{(0)}\times G^{(0)}$. Therefore, $P$ is closed in $S\times G \times T$, and therefore it is locally compact.
\end{proof}

The following examples show that the weak pullback of groupoids is a natural notion. A more detailed study of these examples and many others will appear in a separate paper, where we discuss the weak pullback in the context of topological and measure theoretic degroupoidification.

\begin{example}\label{ex:open cover WPB} \emph{(weak pullback of open cover groupoids)}
\end{example}
Let $X$, $Y$ and $Z$ be locally compact topological spaces, and let $p:Y \rightarrow X$ and $q:Z\rightarrow X$ be continuous, open and surjective maps. Assume that $\mathcal{U}= \{U_\alpha\}_{\alpha \in A}$ and $\mathcal{W}= \{W_\alpha\}_{\alpha \in A}$ are locally finite open covers of $Y$ and $Z$, respectively (with the same indexing set $A$), and assume that $p(U_\alpha) = q(W_\alpha)$ for every $\alpha \in A$, defining an open cover $\mathcal{V}= \{V_\alpha\}_{\alpha \in A}$ of $X$, where $V_\alpha = p(U_\alpha)$. Consider the regular pullback diagram in the category \textbf{\textit{Top}} of topological spaces and continuous functions:
\[
\xymatrix{
& & Y\!*\!Z \ar[dl]_{\pi_Y} \ar[dr]^{\pi_Z} & &\\
& Y \ar[dr]_{p} & & Z \ar[dl]^{q} & \\
 & & X & &
}
\]
where $Y\!*\!Z = \{(y,z) \in Y\!\times\!Z ~|~ p(y)=q(z)\}$. All sets of the form $(U_\alpha \!\times\! W_\beta) \cap Y\!*\!Z$ constitute an open cover of the pullback space $Y\!*\!Z$, which we will denote by $\mathcal{U}\!*\!\mathcal{W}$.

Associated to an open cover $\mathcal{U}$ of a space $Y$ is a groupoid $\GU = \{ (\alpha, y, \beta): y \in U_\alpha \cap U_\beta \}$ (called an open cover groupoid, or $\check{C}ech$ groupoid). A pair $(\alpha, y, \beta)$, $(\gamma, y', \delta)$ is composable if and only if $\beta=\gamma$ and $y=y'$, in which case their product is $(\alpha, y, \delta)$, and the inverse is given by $(\alpha, y, \beta)^{-1}= (\beta, y, \alpha)$. Let $\GU$, $\mathcal{G}_{\mathcal{W}}$ and $\mathcal{G}_{\mathcal{V}}$ be the open cover groupoids associated to the covers of $Y$, $Z$ and $X$ above, and let $\widehat{p}:\GU \rightarrow \mathcal{G}_{\mathcal{V}}$ and $\widehat{q}:\mathcal{G}_{\mathcal{W}} \rightarrow \mathcal{G}_{\mathcal{V}}$ be the induced homomorphisms, given by $\widehat{p}(\alpha,y,\beta) = (\alpha, p(y), \beta)$ and $\widehat{q}(\alpha,z,\beta) = (\alpha, q(z), \beta)$. This gives rise to a cospan diagram of groupoids, which can be completed to a weak pullback diagram:
\[
\xymatrix{
& & \mathcal{P} \ar[dl] \ar[dr] & &\\
& \GU \ar[dr]_{\widehat{p}} & & \mathcal{G}_{\mathcal{W}} \ar[dl]^{\widehat{q}} & \\
 & & \mathcal{G}_{\mathcal{V}} & &
}
\]
We omit the technical but straightforward calculations which yield the upshot: the \emph{weak pullback groupoid} $\mathcal{P}$ is isomorphic to the open cover groupoid $\mathcal{G}_{\mathcal{U}*\mathcal{W}}$ corresponding to the cover $\mathcal{U}\!*\!\mathcal{W}$ of the \emph{regular pullback space} $Y\!*\!Z$.

\begin{example}\label{ex:transformation groupoid WPB} \emph{(weak pullback of transformation groupoids)}
\end{example}
Let $X$, $Y$ and $Z$ be locally compact topological spaces, and let $p:Y \rightarrow X$ and $q:Z\rightarrow X$ be continuous maps. Let $Y\!*\!Z$ be the regular pullback in the category \textbf{\textit{Top}}, as in the previous example. Let $\Gamma$ and $\Lambda$ be locally compact groups acting on $Y$ and $Z$ respectively, and let $Y \!\times\! \Gamma$ and $Z \!\times\! \Lambda$ be the corresponding transformation groupoids. Recall that in a transformation groupoid, say $Y \!\times\! \Gamma$, the elements $(y,\gamma)$ and $(\tilde{y},\tilde{\gamma})$ are composable if and only if $\tilde{y} =y\gamma$, in which case $(y,\gamma)(y\gamma,\tilde{\gamma}) = (y,\gamma\tilde{\gamma})$. The inverse, range and domain are given by $(y,\gamma)^{-1} = (y\gamma, \gamma^{-1})$, $r(y,\gamma) = (y,e)$ and $d(y,\gamma) = (y\gamma,e)$.

We view $X$ as a transformation groupoid by endowing it with an action of the trivial group, which amounts to regarding $X$ as a cotrivial groupoid. Assume that the maps $p$ and $q$ are equivariant with respect to the group actions, i.e. $p(y \cdot \gamma) = p(y)$ and $q(z \cdot \lambda) = q(z)$. In this case $p$ and $q$ induce groupoid homomorphisms $\hat{p}:Y\!\times\!\Gamma \rightarrow X$ and $\hat{q}:Z\!\times\!\Lambda \rightarrow X$ given by $\hat{p}(y,\gamma) = p(y)$ and $\hat{q}(z,\lambda)=q(z)$. This yields a cospan diagram of topological groupoids which gives rise to the following weak pullback diagram:
\[
\xymatrix{
& & \mathcal{P} \ar[dl]_{\pi_Y} \ar[dr]^{\pi_Z} & &\\
& Y \!\times\! \Gamma \ar[dr]_{\hat{p}} & & Z\!\times\! \Lambda \ar[dl]^{\hat{q}} & \\
 & & X & &
}
\]
It is now not hard to verify that the \emph{weak pullback groupoid} $\mathcal{P}$ can be identified with the transformation groupoid $(Y\!*\!Z) \!\times\! (\Gamma \!\times\! \Lambda)$ corresponding to the action of the group $(\Gamma \!\times\! \Lambda)$ on the \emph{regular pullback space} $(Y\!*\!Z)$, given by $(y,z) \cdot (\gamma,\lambda) = (y\gamma, z\lambda)$. \\

\begin{rem}\label{rem:weak and regular pullback conincide for cotrivial G}
In general, the weak pullback coincides with a regular pullback whenever the groupoid $G$ in Definition \ref{def:weak pullback} is a cotrivial groupoid. This is the case in example \ref{ex:transformation groupoid WPB} above.
\end{rem}

The following observation will be essential in the sequel.
\begin{lem}\label{lem:fibers of P}
For any $u=(s,g,t) \in P^{(0)}$, the fiber $P^u$ is a cartesian product of the form $P^u = P^{(s,g,t)} = S^s \times \{g\} \times T^t$.
\end{lem}

\begin{proof}
We follow the definitions:
\begin{eqnarray}
P^{(s,g,t)} &=& \{(\sigma, h, \tau) \in P ~|~ r_P(\sigma, h, \tau) = (s,g,t)\} \nonumber \\
&=& \{(\sigma, h, \tau) \in P ~|~ (r_S(\sigma), h, r_T(\tau)) = (s,g,t) \} \nonumber \\
&=& \{(\sigma, h, \tau) \in P ~|~ r_S(\sigma) = s, h=g, r_T(\tau) = t \} \nonumber \\
&=& \{(\sigma, h, \tau) \in P ~|~ \sigma \in S^s, h=g, \tau \in T^t \}. \nonumber
\end{eqnarray}
Note that since $(s,g,t)$ is an element of $P^{(0)}$, any $\sigma \in S^s$ satisfies $r_G(p(\sigma)) = p(r_S(\sigma)) = p(s) = p(r_S(s)) = r_G(p(s)) = r_G(g)$ and likewise any $\tau \in T^t$ satisfies $r_G(q(\tau)) = d_G(g)$. Therefore $S^s \times \{g\} \times T^t \subseteq P $ and thus
$$P^{(s,g,t)} \ = \ \{(\sigma, h, \tau) \in P ~|~ \sigma \in S^s, h=g, \tau \in T^t \} \ = \ S^s \times \{g\} \times T^t.$$
\end{proof}

\begin{prop}\label{prop:piS and piT continuous homomorphisms}
The projections $\pi_S : P \rightarrow S$ and $\pi_T : P \rightarrow T$ are continuous groupoid homomorphisms.
\end{prop}

\begin{proof}
The proof is straightforward. For continuity, let $A \subseteq S$ be an open subset. Then  $\pi_S^{-1}(A)$ is open in $P$ since $\pi_S^{-1}(A) = \{ (s,g,t) \in P ~|~ \pi_S(s,g,t) \in A \} = \{ (s,g,t) \in P ~|~ s \in A \} = (A \times G \times T) \cap P$. Now take $((s,g,t),(\sigma,h,\tau)) \in P^{(2)}$. Then $\pi_S ((s,g,t)(\sigma,h,\tau)) = \pi_S(s \sigma ,g,t \tau) = s \sigma = \pi_S (s,g,t) \pi_S (\sigma,h,\tau)$. Also, $\pi_S((s,g,t)^{-1}) = \pi_S(s^{-1},p(s)^{-1} gq(t),t^{-1} ) = s^{-1} = (\pi_S(s,g,t))^{-1}$. Thus $\pi_S$ is a groupoid homomorphism. The proof for $\pi_T$ is similar.
\end{proof}

%***************************************************************
%***************************************************************
\section{A Haar system for the weak pullback}\label{sec:Haar system for wpb}
%***************************************************************
%***************************************************************

We now assume that $S$, $G$ and $T$ are Haar groupoids and that the maps $p$ and $q$ are homomorphisms of Haar groupoids. In order to define the weak pullback of the following diagram in the category $\mathcal{HG}$, we let $P$ be the weak pullback of the underlying diagram of topological groupoids, as defined above.
\[
\begin{array}{cccc}
\xymatrix{& & P \ar[dl] \ar[dr] & \\ \scriptstyle{\lambda_S^{\bullet}, \ \mu_S^{(0)}} \!\!\!\!\!\!\!\!\!\!\!\!\!\!\! & S \ar[dr]_p & & T \ar[dl]^q
& \!\!\!\!\!\!\!\!\!\!\!\! \scriptstyle{\lambda_T^{\bullet}, \ \mu_T^{(0)}} \\
& & G & \!\!\!\!\!\!\!\!\!\!\!\!\! \scriptstyle{\lambda_G^{\bullet}, \ \mu_G^{(0)}} \\}
\end{array}
\]

Our goal is to construct a Haar groupoid structure on $P$. We start by defining the Haar system $\lambda_P^{\bullet}$. From Lemma \ref{lem:fibers of P} we know that the $r$-fibers of $P$ are cartesian products of the form $P^u = P^{(s,g,t)} = S^s \times \{g\} \times T^t$. In light of this it is reasonable to propose the following definition.

\begin{mydef}\label{def:Haar for P}
Let $u = (s,g,t) \in P^{(0)}$. Define $$\lambda_P^{u} = \lambda_P^{(s,g,t)} := \lambda_S^{s} \times \delta_g \times \lambda_T^{t}.$$ We denote $\lambda_P^{\bullet} = \{ \lambda_P^{u} \}_{u \in P^{(0)}}$.
\end{mydef}

\begin{theorem}\label{thm:Haar system for P}
The system $\lambda_P^{\bullet}$ is a continuous left Haar system for $P$.
\end{theorem}

\begin{proof}
The proof will rely on the technology developed in \cite{BSM}. We consider the following three pullback diagrams in the category \textbf{\textit{Top}} of topological spaces and continuous functions (i.e. we temporarily forget the algebraic structures of the groupoids involved, and view them only as topological spaces. Likewise all groupoid homomorphisms are regarded only as continuous functions):

\[
\begin{array}{cc}
\mathbf{Diagram \ A} \ \ \ \ & \xymatrix{G^{(0)} * G^{(0)} \ar [dd]\ar [rr]&& G^{(0)} \ar [dd]^{v \mapsto [v]}\\\\
G^{(0)}\ar [rr]_{u \mapsto [u]}&&\underline{G}}
\end{array}
\]

\[
\begin{array}{cc}
\mathbf{Diagram \ B} \ \ \ \ & \xymatrix{S^{(0)} * T^{(0)} \ar [dd]\ar [rr]&& T^{(0)} \ar [dd]^{t \mapsto [q(t)]}\\\\
S^{(0)}\ar [rr]_{s \mapsto [p(s)]}&&\underline{G}}
\end{array}
\]

\[
\begin{array}{cc}
\mathbf{Diagram \ C} \ \ \ \ \ \ \ & \xymatrix{S * T \ar [dd]\ar [rr]&& T \ar [dd]^{\tau \mapsto [q(r(\tau))]}\\\\
S\ar [rr]_{\sigma \mapsto [p(r(\sigma))]}&&\underline{G}}
\end{array}
\]
Note that in order to lighten notation, we denote the pullback object, for example in Diagram C, by $S * T$ in place of $S *_{\underline{G}} T$. By definition $$S * T = S *_{\underline{G}} T = \{ (\sigma,\tau) \in S \times T ~|~ [p(r(\sigma))] = [q(r(\tau))] \text { in } \underline{G} \} $$ and the maps to $S$ and $T$ are the obvious projections. The topology of $S * T$ is the restriction of the product topology on $S \times T$.

Using $G^{(0)} * G^{(0)}$, $S^{(0)} * T^{(0)}$ and $S * T$, we can now construct two more pullback diagrams (still in \textbf{\textit{Top}}). Our identifications of the pullback objects in Diagrams D and E with $P^{(0)}$ and $P$, respectively, are justified below. A moment's reflection reveals that the maps in these diagrams are well defined.

\[
\begin{array}{cc}
\ \ \ \ \ \ \ \mathbf{Diagram \ D} \ \ \ \ & \ \ \ \ \xymatrix{P^{(0)} \ar [dd]\ar [rr]&& S^{(0)} * T^{(0)} \ar [dd]^{(s,t) \mapsto (p(s),q(t))}\\\\
G\ar [rr]_{\hspace{-.3in} x \mapsto (r(x),d(x))}&&G^{(0)} * G^{(0)}}
\end{array}
\]

\[
\begin{array}{cc}
\ \ \ \ \ \ \ \ \ \ \ \ \ \mathbf{Diagram \ E} \ \ \ \ & \ \ \ \ \ \ \ \xymatrix{P \ar [dd]\ar [rr]&& S * T \ar [dd]^{(\sigma,\tau) \mapsto (p(r(\sigma)),q(r(\tau)))}\\\\
G\ar [rr]_{\hspace{-.3in} x \mapsto (r(x),d(x))}&&G^{(0)} * G^{(0)}}
\end{array}
\]

\noindent In Diagram D we identified the pullback object $G *_{G^{(0)} *_{\underline{G}} G^{(0)}} (S^{(0)} *_{\underline{G}} T^{(0)})$ with $P^{(0)}$. Indeed, \begin{eqnarray}
G *_{G^{(0)} *_{\underline{G}} G^{(0)}} (S^{(0)} *_{\underline{G}} T^{(0)}) &=& \{ (g, (s,t)) ~|~ (r_G(g),d_G(g)) = (p(s),q(t)) \}\nonumber \\
&=& \{ (g, (s,t)) ~|~ r_G(g) = p(s) \text{ and } d_G(g) = q(t) \} \nonumber \\
&=& \{ (g, (s,t)) ~|~  g \in G^{p(s)}_{q(t)} \} \nonumber
\end{eqnarray}
which can obviously be identified, as sets, with our definition of $P^{(0)}$. Moreover, the topology on the pullback is precisely that of  $P^{(0)}$, namely the induced topology from $S^{(0)}\! \times \!G \!\times\! T^{(0)}$. Similarly, in Diagram E we identified the pullback object $G *_{G^{(0)} *_{\underline{G}} G^{(0)}} (S *_{\underline{G}} T)$ with $P$. Indeed,
\begin{eqnarray}
G *_{G^{(0)} *_{\underline{G}} G^{(0)}} (S *_{\underline{G}} T) &=& \{ (g, (s,t)) ~|~ (r_G(g),d_G(g)) = (p(r_S(s)),q(r_T(t))) \}\nonumber \\
&=& \{ (g, (s,t)) ~|~ r_G(g) = p(r_S(s)) \text{ and } d_G(g) = q(r_T(t)) \} \nonumber
\end{eqnarray}
which can be identified with our definition of $P$, as sets as well as in \textbf{\textit{Top}}.

Henceforth, we shall follow Section 5 of \cite{BSM}, where we studied fibred products of systems of measures. Observe that the results we invoke at this point from \cite{BSM} only require spaces to be $\mathbf{T_1}$ and second countable. The spaces we consider all satisfy these hypotheses. Using Diagram C as the front face and Diagram B as the back face, we construct the following fibred product diagram:
\[
\xy 0;<.2cm,0cm>:
(20,20)*{S^{(0)} * T^{(0)}}="1";
(40,20)*{T^{(0)}}="2";
(20,0)*{S^{(0)}}="3";
(40,0)*{\underline{G}}="4";
(10,10)*{S * T}="5";
(30,10)*{T}="6";
(10,-10)*{S}="7";
(30,-10)*{\underline{G}}="8";
{"1"+CR+(.5,0);"2"+CL+(-.5,0)**@{-}?>*{\dir{>}}};
{"5"+CR+(.5,0);"6"+CL+(-.5,0)**@{-}?>*{\dir{>}}};
{"3"+CR+(.5,0);"4"+CL+(-10,0)**@{-}};
{"3"+CR+(9,0);"4"+CL+(-.5,0)**@{-}?>*{\dir{>}}?>(.5)+(0,1)*{\scriptstyle [p]}};
{"7"+CR+(.5,0);"8"+CL+(-.5,0)**@{-}?>*{\dir{>}}?>(.5)+(0,-1)*{\scriptstyle [p\circ r_{S}]}};
{"1"+CD+(0,-.5);"3"+CU+(0,9.6)**@{-}} ;
{"1"+CD+(0,-9.8);"3"+CU+(0,.5)**@{-}?>*{\dir{>}}} ;
{"2"+CD+(0,-.5);"4"+CU+(0,.5)**@{-}?>*{\dir{>}} ?>(.5)+(1,0)*{\scriptstyle [q]}} ;
{"5"+CD+(0,-.5);"7"+CU+(0,.5)**@{-}?>*{\dir{>}}} ;
{"6"+CD+(0,-.5);"8"+CU+(0,.5)**@{-}?>*{\dir{>}} ?>(.75)+(-2.3,0)*{\scriptstyle [q\circ r_{T}]}} ;
{"5"+C+(1.4,1.4);"1"+C+(-1.4,-1.4)**@{-}?>*{\dir{>}}?>(.5)+(-2.3,0)*{\scriptstyle r_S * r_T}} ;
{"6"+C+(1.4,1.4);"2"+C+(-1.4,-1.4)**@{-}?>*{\dir{>}} ?>(.5)+(-1.5,0)*{\scriptstyle r_T}?>(.5)+(2,0)*{\scriptstyle \lambda_T^{\bullet}}} ;
{"7"+C+(1.4,1.4);"3"+C+(-1.4,-1.4)**@{-}?>*{\dir{>}} ?>(.5)+(-1.5,0)*{\scriptstyle r_S}?>(.5)+(2,0)*{\scriptstyle \lambda_S^{\bullet}}} ;
{"8"+C+(1.4,1.4);"4"+C+(-1.4,-1.4)**@{-}?>*{\dir{>}}};
\endxy
\]
The connecting maps are the range maps $r_T$ and $r_S$, and they are endowed respectively with the Haar systems $\lambda_T^{\bullet}$ and $\lambda_S^{\bullet}$, which are continuous systems of measures and therefore locally finite (see Corollary 2.15 of \cite{BSM}). It is immediate to see that the compatibility conditions on the maps of the bottom and the right faces are satisfied. The map $r_S * r_T : S * T \rightarrow S^{(0)} * T^{(0)}$ is defined by $(r_S * r_T)(s,t)=(r_S(s), r_T(t))$. By Definition 5.1 and Proposition 5.2 of \cite{BSM}, we obtain a locally finite system of measures
$(\lambda_S * \lambda_T)^{\bullet}$ on $r_S * r_T,$ where $$(\lambda_S * \lambda_T)^{(s,t)} = \lambda_S^s \times \lambda_T^t.$$ Moreover,
by Proposition 5.5 of \cite{BSM} it is positive on open sets.

With this at hand, we construct another fibred product diagram. We take Diagram E as the front face and Diagram D as the back face, and use $r_S * r_T$ and $id:G \rightarrow G$ as the connecting maps. The map $r_S * r_T$ is equipped with the above locally finite system of measures $(\lambda_S * \lambda_T)^{\bullet}$, whereas the identity map on $G$ naturally admits the system $\delta^{\bullet}$ of Dirac masses, which is trivially locally finite:
\[
\xy 0;<.2cm,0cm>:
(20,20)*{P^{(0)}}="1";
(40,20)*{S^{(0)} * T^{(0)}}="2";
(20,0)*{G}="3";
(40,0)*{G^{(0)} * G^{(0)}}="4";
(10,10)*{P}="5";
(30,10)*{S * T}="6";
(10,-10)*{G}="7";
(30,-10)*{G^{(0)} * G^{(0)}}="8";
{"1"+CR+(.5,0);"2"+CL+(-.5,0)**@{-}?>*{\dir{>}}};
{"5"+CR+(.5,0);"6"+CL+(-.5,0)**@{-}?>*{\dir{>}}};
{"3"+CR+(.5,0);"4"+CL+(-6.5,0)**@{-}?>(.5)+(0,-1)*{\scriptstyle{(r,d)}}};
{"3"+CR+(10,0);"4"+CL+(-.5,0)**@{-}?>*{\dir{>}}};
{"7"+CR+(.5,0);"8"+CL+(-.5,0)**@{-}?>*{\dir{>}}?>(.5)+(0,-1)*{\scriptstyle (r,d)}};
{"1"+CD+(0,-.5);"3"+CU+(0,9.6)**@{-}} ;
{"1"+CD+(0,-9.8);"3"+CU+(0,.5)**@{-}?>*{\dir{>}}} ;
{"2"+CD+(0,-.5);"4"+CU+(0,.5)**@{-}?>*{\dir{>}}?>(.5)+(1.45,0)*{\scriptstyle p * q}} ;
{"5"+CD+(0,-.5);"7"+CU+(0,.5)**@{-}?>*{\dir{>}}} ;
{"6"+CD+(0,-.5);"8"+CU+(0,.5)**@{-}?>*{\dir{>}}?>(.5)+(4,4)*{\scriptstyle (p\circ r) * (q\circ r)}} ;
{"5"+C+(1.4,1.4);"1"+C+(-1.4,-1.4)**@{-}?>*{\dir{>}}?>(.5)+(0,1)*{\scriptstyle r_P}} ;
{"6"+C+(1.4,1.4);"2"+C+(-1.4,-1.4)**@{-}?>*{\dir{>}}?>(.5)+(-3,0)*{\scriptstyle r_S * r_T}?>(.5)+(4,0)*{\scriptstyle (\lambda_S * \lambda_T)^{\bullet}}} ;
{"7"+C+(1.4,1.4);"3"+C+(-1.4,-1.4)**@{-}?>*{\dir{>}}?>(.5)+(-1.5,0)*{\scriptstyle id}?>(.5)+(1.5,0)*{\scriptstyle \delta^{\bullet}}} ;
{"8"+C+(1.4,1.4);"4"+C+(-1.4,-1.4)**@{-}?>*{\dir{>}}} ;
\endxy
\]
It is again easy to see that the compatibility conditions on the maps of the bottom and the right faces are satisfied. Note that in this last diagram we have identified the map from $P$ to $P^{(0)}$ with $r_P$, the range map of $P$.

Resorting once again to Definition 5.1 and Proposition 5.2 of \cite{BSM}, we obtain a locally finite system of measures $(\delta * (\lambda_S *
\lambda_T))^{\bullet}$ on $r_P : P \rightarrow P^{(0)}$, where $$(\delta *(\lambda_S *\lambda_T))^{(g,s,t)} = \delta_g \times (\lambda_S * \lambda_T)^{(s,t)} = \delta_g \times \lambda_S^s \times \lambda_T^t.$$ We denote this system of measures on $r_P$ by $\lambda_P^{\bullet}$. Yielding to the original convention of writing elements of $P$ as $(s,g,t)$ rather than $(g,s,t)$, we write $\lambda_P^{(s,g,t)} = \lambda_S^s \times \delta_g \times \lambda_T^t$. Our construction of $\lambda_P^{\bullet}$ as a fibred product of the systems $\delta^{\bullet}$ and $(\lambda_S *\lambda_T)^{\bullet}$, which are locally finite and positive on open sets, guarantees (by Propositions 5.2 and 5.5 of \cite{BSM}) that $\lambda_P^{\bullet}$ inherits these properties.

Recall that as we have pointed out in the preliminaries, $\underline{G}$ need not be a Hausdorff space in general. Moreover, $S * T$, for example, need not be locally compact, as it is not necessarily closed in $S \times T$. The assumption that all spaces are locally compact and Hausdorff is essential in the CSM setting in \cite{BSM}. For this reason we cannot simply use Proposition 5.4 of \cite{BSM} to deduce that as fibred products, $(\lambda_S * \lambda_T)^{\bullet}$ and subsequently $\lambda_P^{\bullet}$ are CSMs. Thus, we present a separate direct proof that $\lambda_P^{\bullet}$ is a CSM in Proposition \ref{prop:lambdaP is a CSM} below. Furthermore, at this point we return to viewing $P$, $G$, $S$ and $T$ as groupoids, and in Proposition \ref{prop:left invariance of lambda P} we state and prove that $\lambda_P^{\bullet}$ is left invariant. We conclude that $\lambda_P^{\bullet}$ is a continuous left Haar system for the groupoid $P$.
\end{proof}

\begin{prop}\label{prop:lambdaP is a CSM}
The system $\lambda_P^{\bullet}$ is a continuous system of measures.
\end{prop}

\begin{proof}
From the definition of a CSM, in order to prove that $\lambda_P^{\bullet}$ is a CSM on $r_P:P \rightarrow P^{(0)}$, we need to show that for any $0 \leq f\in C_c(P)$, the map $ (s,g,t) \mapsto \int_P f(\sigma,x,\tau)d \lambda_P^{(s,g,t)} (\sigma,x,\tau)$ is a continuous function on $P^{(0)}$.

Let $0 \leq f\in C_c(P)$. Recall from the proof of Lemma \ref{lem:P is LC+T2+2nd} that $P$ is closed in $S\!\times\! G \!\times\! T$. By Tietze's Extension Theorem, there exists a function $F\in C(S\!\times\! G \!\times\! T)$ such that $F|_{_P}=f$. Since we can multiply $F$ by a function $\varphi \in C_c(S\!\times\! G \!\times\! T)$ which satisfies $\varphi =1$ on $K = supp(f)$, we can assume, without loss of generality, that $F \in C_c(S\!\times\! G \!\times\! T)$.

We now resort to (symmetric versions of) Lemma 4.5 in \cite{BSM}. First we take $\mathcal{X} \!=\! S \!\times\! G$, $\mathcal{Y} \!=\! T$, $\mathcal{Z} \!=\! T^{(0)}$ and $\gamma^{\bullet}\! = \!\lambda_T^{\bullet}$, to deduce that the function $F_1$ defined by $ (\sigma,x,t) \mapsto \int_T F(\sigma,x,\tau)d \lambda_T^t (\tau)$ is in $C_c(S \!\times\! G \!\times\! T^{(0)})$. Next, taking $\mathcal{X} \!=\! S \!\times\! T^{(0)}$, $\mathcal{Y} \!=\! G$, $\mathcal{Z} \!=\! G$ and $\gamma^{\bullet} \!=\! \delta^{\bullet}$, we get that the function $F_2$ defined by $ (\sigma,g,t) \mapsto \int_G F_1(\sigma,x,t)d \delta_g (x)$ is in $C_c(S \!\times\! G \!\times\! T^{(0)})$. Finally, with $\mathcal{X} \!=\! G \!\times\! T^{(0)}$, $\mathcal{Y} \!=\! S$, $\mathcal{Z} \!=\! S^{(0)}$ and $\gamma^{\bullet} \!= \! \lambda_S^{\bullet}$, Lemma 4.5 of \cite{BSM} implies that the function $F_3$ defined by $ (s,g,t) \mapsto \int_S F_2(\sigma,g,t)d \lambda_S^{s} (\sigma)$ is in $C_c(S^{(0)}\! \times \!G \!\times\! T^{(0)})$. Merging these results, we can rewrite the function $F_3$ by $$(s,g,t)\longmapsto \int_S\int_G\int_T F(\sigma,x,\tau)\ d\lambda_T^t(\tau)d\delta_g(x)\lambda_S^s(\sigma). $$

Note that in the above integral $r_S(\sigma) = s$ and $r_T(\tau) = t$, since $supp(\lambda_S^{s}) = r_S^{-1}(s)$ and $supp(\lambda_T^{t}) = r_T^{-1}(t)$.
Therefore, if we take $(s,g,t) \in P^{(0)}$, in which case $p(s) = r_G(g)$ and $q(t) = d_G(g)$, we get that $p(r_S(\sigma)) = r_G(g)$ and $q(r_T(\tau)) = d_G(g)$. In other words, when restricting $F_3$ to $P^{(0)}$, we are actually integrating over $P$. Recalling the definition of $\lambda_P^{\bullet}$ and that $F|_{_P}=f$, we retrieve precisely the function $(s,g,t) \mapsto \int_P f(\sigma,x,\tau)d \lambda_P^{(s,g,t)} (\sigma,x,\tau)$, which is continuous on $P^{(0)}$ as a restriction of a continuous function on $S^{(0)}\! \times \!G \!\times\! T^{(0)}$.
\end{proof}

\begin{prop}\label{prop:left invariance of lambda P}
The system $\lambda^{\bullet}_P$ is left invariant.
\end{prop}

\begin{proof}
From the definition of left invariance, we need to show that
\begin{equation}\label{eq:left inv for P}
\lambda_P^{d_P(x)} (E) = \lambda_P^{r_P(x)}\left(x \cdot (E \cap P^{d_P(x)})\right),
\end{equation}
for every $x \in P$ and for every Borel subset $E \subseteq P$.

Assume first that $E$ is a set of the form $E\!= \!(A \times B \times C) \cap P$, where $A \!\subseteq \!S$, $B \!\subseteq \! G$ and $C\! \subseteq \! T$. Let $x=(\sigma,y,\tau) \in P$, so $r_P(x) = (r_S(\sigma),y,r_T(\tau))$ and $d_P(x) = (d_S(\sigma),p(\sigma)^{-1}yq(\tau),d_T(\tau))$. We will denote $z = p(\sigma)^{-1}yq(\tau)$. We calculate the left and right hand sides of (\ref{eq:left inv for P}) separately. On the one hand we get:
\begin{eqnarray}
\lambda_P^{d_P(x)} (E) &=& \lambda_P^{d_P(x)} \left((A \times B \times C) \cap P^{d_P(x)}\right) \qquad \text{ since } \lambda_P^{d_P(x)} \text{ is concentrated on } P^{d_P(x)} \nonumber \\
&=& \lambda_P^{d_P(x)} \left((A \times B \times C) \cap (S^{d_S(\sigma)} \times \{z\} \times T^{d_T(\tau)})\right) \qquad \text{ by Lemma \ref{lem:fibers of P}} \nonumber \\
&=& \lambda_P^{d_P(x)} \left((A \cap S^{d_S(\sigma)}) \times (B \cap \{z\}) \times (C \cap T^{d_T(\tau)})\right) \nonumber \\
&=& \lambda_S^{d_S(\sigma)} (A \cap S^{d_S(\sigma)}) \cdot \delta_z (B \cap \{z\}) \cdot \lambda_T^{d_T(\tau)} (C \cap T^{d_T(\tau)}) \nonumber \\
&=& \lambda_S^{d_S(\sigma)} (A) \cdot \delta_z (B) \cdot \lambda_T^{d_T(\tau)} (C) \nonumber
\end{eqnarray}
On the other hand,
\begin{eqnarray}
\lambda_P^{r_P(x)}\left(x \cdot (E \cap P^{d_P(x)})\right) &=& \lambda_P^{r_P(x)} \left((\sigma,y,\tau) \cdot \left((A \times B \times C) \cap P^{d_P(x)}\right)\right) \nonumber \\
&=& \lambda_P^{r_P(x)} \left((\sigma,y,\tau) \cdot \left((A \cap S^{d_S(\sigma)}) \times (B \cap \{z\}) \times (C \cap T^{d_T(\tau)})\right)\right) \nonumber
\end{eqnarray}
By the definition of $P^{(2)}$, note that $(\sigma,y,\tau) \cdot \left((A \cap S^{d_S(\sigma)}) \times (B \cap \{z\}) \times (C \cap T^{d_T(\tau)})\right)$ can be nonempty only when $z = p(\sigma)^{-1}yq(\tau) \in B$, in which case
the middle component of the product is $\{y\}$. Hence
\begin{eqnarray}
&=& \begin{cases} \lambda_P^{r_P(x)} \left(\sigma \cdot (A \cap S^{d_S(\sigma)}) \times  \{y\} \times \tau \cdot (C \cap T^{d_T(\tau)})\right) & z \in B \\ \lambda_P^{r_P(x)} (\emptyset) & z \notin B \end{cases} \nonumber \\
&=& \begin{cases} \lambda_S^{r_S(\sigma)} \left(\sigma \cdot (A \cap S^{d_S(\sigma)})\right) \cdot \delta_y (\{y\}) \cdot \lambda_T^{r_T(\tau)} \left( \tau \cdot (C \cap T^{d_T(\tau)})\right) & z \in B \\ 0 & z \notin B \end{cases} \nonumber \\
&=& \begin{cases} \lambda_S^{d_S(\sigma)} (A) \cdot \lambda_T^{d_T(\tau)} (C) & z \in B \\ 0 & z \notin B \end{cases} \qquad \text{ by the left invariance of } \lambda_S^{\bullet} \text{ and } \lambda_T^{\bullet} \nonumber \\
&=& \lambda_S^{d_S(\sigma)} (A) \cdot \delta_z (B)  \cdot \lambda_T^{d_T(\tau)} (C) \nonumber
\end{eqnarray}
Thus (\ref{eq:left inv for P}) holds for any set $E$ of the form $E= (A \times B \times C) \cap P$.

Fix $x\in P$, and for any Borel subset $E$ of $P$ define
$$\mu(E)=\lambda_P^{d_P(x)} (E) \qquad \text{and} \qquad \nu(E)=\lambda_P^{r_P(x)}\left(x \cdot (E \cap P^{d_P(x)})\right).$$ We claim that $\mu$ and $\nu$ are both locally finite measures on $P$. Since $\lambda^{\bullet}_P$ is a CSM, it is a locally finite BSM by Proposition 2.23 of \cite{BSM}. Hence $\lambda_P^{u}$ is a locally finite measure for any $u \in P^{(0)}$, and in particular $\mu=\lambda_P^{d_P(x)}$ is a locally finite measure.

We turn to $\nu$. It is trivial that $\nu(\emptyset) = 0$. Let $\{E_i\}_{i=1}^{\infty}$ be a countable collection of disjoint Borel subsets of $P$.
$$\nu(\bigcup_{i=1}^{\infty}E_{i}) \ = \ \lambda_P^{r_P(x)}(x \cdot ((\bigcup_{i=1}^{\infty}E_{i}) \cap P^{d_P(x)})) \ = \ \lambda_P^{r_P(x)}(x \cdot (\bigcup_{i=1}^{\infty}(E_{i} \cap P^{d_P(x)}))) \ = $$$$ = \ \lambda_P^{r_P(x)}(\bigcup_{i=1}^{\infty}x \cdot(E_{i} \cap P^{d_P(x)})) \ = \ \sum_{i=1}^{\infty}\lambda_P^{r_P(x)}\left(x \cdot\left(E_{i} \cap P^{d_P(x)}\right)\right) \ = \ \sum_{i=1}^{\infty}\nu(E_{i}).$$ Therefore $\nu$ is countably additive, and hence a measure. In order to prove that $\nu$ is locally finite we need to show that every $y\in P$ admits an open neighborhood $U_y$ such that $\nu(U_y)<\infty$. In the case where $y\notin P^{d_P(x)}$, the open set $U_y=P\setminus P^{d_P(x)}$ satisfies $\nu(U_y)=\lambda_P^{r_P(x)}\left(x \cdot (U_y \cap P^{d_P(x)})\right)=\lambda_P^{r_P(x)}\left(\emptyset\right)=0<\infty$. Now assume that $y\in P^{d_P(x)}$. In this case the product $z=xy$ is well defined, and since $\lambda_P^{r_P(x)}$ is a locally finite measure, there exists an open neighborhood $U_z$ of $z$ such that $\lambda_P^{r_P(x)}(U_z)<\infty$. The map $P^{d_P(x)}\rightarrow P$ defined by $w\mapsto x\cdot w$ is continuous, hence there exists an open neighborhood $U_y$ of $y$ such that $x\cdot \left(U_y\cap P^{d_P(x)}\right)\subset U_z$. Consequently, $\nu(U_y)=\lambda_P^{r_P(x)}\left(x \cdot (U_y \cap P^{d_P(x)})\right)\leq \lambda_P^{r_P(x)}\left(U_z\right)<\infty$.

Finally, let ${\mathcal B}$ be a countable basis for the topology of $P$ consisting of elementary open sets, as in Remark \ref{rem:elementary open subsets of P}.  As we have just shown, elementary open sets satisfy (\ref{eq:left inv for P}), hence $\mu$ and $\nu$ agree on all finite intersections of sets in ${\mathcal B}$. We can now invoke Lemma 2.24 of \cite{BSM}, which states that if $\mu$ and $\nu$ are two locally finite measures on a space $X$, and there exists a countable basis ${\mathcal B}$ for the topology of $X$ such that $\mu(U_{1}\cap U_{2}\cap\dots\cap U_{n})=\nu(U_{1}\cap U_{2}\cap\dots\cap U_{n})$ for any $\{U_{1},U_{2},\dots, U_{n}\}\subset {\mathcal B}$, $n\geq 1$, then $\mu(E) = \nu(E)$ for any Borel subset $E \subseteq X$. Applying Lemma 2.24 of \cite{BSM} to $\mu$, $\nu$ and ${\mathcal B}$ above completes the proof.
\end{proof}

%***************************************************************
%***************************************************************
\section{A measure on the unit space of the weak pullback}\label{sec:measure for wpb}
%***************************************************************
%***************************************************************

We return to the weak pullback diagram. Our next task is to construct a measure $\mu_P^{(0)}$ on $P^{(0)}$, and for starters we will need to have certain systems of measures $\gamma_p^{\bullet}$ and $\gamma_q^{\bullet}$ on the maps $p$ and $q$, respectively. These systems of measures arise via a disintegration theorem, as we explain below.
\[
\begin{array}{cccc}
\xymatrix{& & P \ar[dl] \ar[dr] & \!\!\!\!\!\!\!\!\!\!\!\!\!\!\!\!\!\!\!\!\!\!\! \scriptstyle{\lambda_P^{\bullet}} \\ \scriptstyle{\lambda_S^{\bullet}, \ \mu_S^{(0)}} \!\!\!\!\!\!\!\!\!\!\!\!\!\!\! & S \ar[dr]_{p,\gamma_p^{\bullet}}
& & T \ar[dl]^{q,\gamma_q^{\bullet}}
& \!\!\!\!\!\!\!\!\!\!\!\! \scriptstyle{\lambda_T^{\bullet}, \ \mu_T^{(0)}} \\ & & G & \!\!\!\!\!\!\!\!\!\!\!\!\! \scriptstyle{\lambda_G^{\bullet}, \ \mu_G^{(0)}} \\}
\end{array}
\]

Let $(X,\mu)$ and $(Y,\nu)$ be measure spaces, and let $f:X \rightarrow Y$ be a Borel map. A system of measures $\gamma^{\bullet}$ on $f$ will be called a \textit{disintegration} (\cite{BSM}, Definition 6.2) of $\mu$ with respect to $\nu$ if $\displaystyle \mu (E) = \int_Y \gamma^y (E) d\nu (y)$ for every Borel set $E \subseteq X$. A \textit{disintegration theorem} gives sufficient conditions which guarantee the existence of such a disintegration, and the version we will use appears as Corollary 6.6 of \cite{BSM}. It requires $\mu$ to be locally finite (and $\sigma$-finite), $\nu$ to be $\sigma$-finite, and $f:X \rightarrow Y$ to be measure class preserving. Under these conditions there exists a locally finite BSM $\gamma^{\bullet}$ on $f$ which is a disintegration of $\mu$ with respect to $\nu$.

Each of the Haar groupoids $S$, $G$ and $T$ is equipped with a Radon (hence locally finite and $\sigma$-finite) measure on its unit spaces, which is quasi-invariant with respect to its Haar system. The maps $p$ and $q$ are homomorphisms of Haar groupoids, therefore $p:S^{(0)} \rightarrow G^{(0)}$ and $q:T^{(0)} \rightarrow G^{(0)}$ are measure class preserving. These ingredients allow us to invoke Corollary 6.6 of \cite{BSM}, and to obtain locally finite BSMs $\gamma_p^{\bullet}$ on $p:S^{(0)} \rightarrow G^{(0)}$ which is a disintegration of $\mu_S^{(0)}$ with respect to $\mu_G^{(0)}$, and $\gamma_q^{\bullet}$ on $q:T^{(0)} \rightarrow G^{(0)}$ which is a disintegration of $\mu_T^{(0)}$ with respect to $\mu_G^{(0)}$.

The following requirement will be essential for our proof of Proposition \ref{prop:muP0 is Radon} below, which states that the measure $\mu_P^{(0)}$ which we are constructing is locally finite.

\begin{ass}\label{ass:LB disintegration}
We will henceforth assume that the disintegration systems $\gamma_p^{\bullet}$ and $\gamma_q^{\bullet}$ can be taken to be locally bounded.
\end{ass}

\begin{rem}\label{rem:LB disintegration}
By Lemma 2.11 of \cite{BSM}, a CSM is always locally bounded. Therefore, an appropriate disintegration theorem that produces a system which is either a CSM or at least locally bounded would have allowed us to remove Assumption \ref{ass:LB disintegration}.
\end{rem}

Continuous (hence locally bounded) disintegrations are abundant: Examples include disintegrations of Lebesgue measures along maps from $\mathbb{R}^n$ to $\mathbb{R}^m$, as well as fiber bundles that admit a continuous disintegration of a measure on the total space with respect to a measure on the base space. Seda shows that more general constructions of fiber spaces also host continuous disintegrations, see Theorem 3.2 of \cite{Seda80}. In our context, a Haar system is of course a continuous disintegration of the induced measure with respect to the measure on the unit space. A very general result (see Theorem 5.43 of \cite{muhly-book-unpublished}, which is a corollary of Theorem 3.3 of \cite{blanchard}) states that any continuous and \emph{open} map $f:X \rightarrow Y$ between second countable locally compact Hausdorff spaces, admits a continuous system of measures $\gamma^{\bullet}$. In particular this implies that if $\nu$ is a measure on $Y$ and we \emph{define} the measure $\mu$ on $X$ via $\gamma^{\bullet}$ by $\displaystyle \mu (E) = \int_Y \gamma^y (E) d\nu (y)$, then $\gamma^{\bullet}$ is a continuous disintegration of $\mu$ with respect to $\nu$.

The next step is to construct a BSM on the projection $\pi_G:P^{(0)} \rightarrow G$, using $\gamma_p^{\bullet}$ and $\gamma_q^{\bullet}$.

\begin{prop}\label{prop:eta}
The projection $\pi_G:P^{(0)} \rightarrow G$ admits a locally finite BSM $\eta^{\bullet}$, given by $$ \eta^x = \gamma_p^{r(x)} \times \delta_x \times \gamma_q^{d(x)} .$$
\end{prop}

\begin{proof}
We form the following fibred product diagram in the category \textbf{\textit{Top}}, with Diagram B as the front face and Diagram A as the back face. The connecting maps are $p:S^{(0)} \rightarrow G^{(0)}$ and $q:T^{(0)} \rightarrow G^{(0)}$, equipped with the locally finite BSMs $\gamma_p^{\bullet}$ and $\gamma_q^{\bullet}$ constructed above. The compatibility conditions on the maps of the bottom and the right faces are easily seen to be satisfied.
\[
\xy 0;<.2cm,0cm>:
(20,20)*{G^{(0)} * G^{(0)}}="1";
(40,20)*{G^{(0)}}="2";
(20,0)*{G^{(0)}}="3";
(40,0)*{\underline{G}}="4";
(10,10)*{S^{(0)} * T^{(0)}}="5";
(30,10)*{ T^{(0)}}="6";
(10,-10)*{S^{(0)}}="7";
(30,-10)*{\underline{G}}="8";
{"1"+CR+(.5,0);"2"+CL+(-.5,0)**@{-}?>*{\dir{>}}};
{"5"+CR+(.5,0);"6"+CL+(-.5,0)**@{-}?>*{\dir{>}}};
{"3"+CR+(.5,0);"4"+CL+(-10.2,0)**@{-}};
{"3"+CR+(9,0);"4"+CL+(-.5,0)**@{-}?>*{\dir{>}}};
{"7"+CR+(.5,0);"8"+CL+(-.5,0)**@{-}?>*{\dir{>}}?>(.5)+(0,-1)*{\scriptstyle [p]}};
{"1"+CD+(0,-.5);"3"+CU+(0,9.6)**@{-}} ;
{"1"+CD+(0,-9.8);"3"+CU+(0,.5)**@{-}?>*{\dir{>}}} ;
{"2"+CD+(0,-.5);"4"+CU+(0,.5)**@{-}?>*{\dir{>}}};
{"5"+CD+(0,-.5);"7"+CU+(0,.5)**@{-}?>*{\dir{>}}};
{"6"+CD+(0,-.5);"8"+CU+(0,.5)**@{-}?>*{\dir{>}}?>(.75)+(-1,0)*{\scriptstyle [q]}} ;
{"5"+C+(1.4,1.4);"1"+C+(-1.4,-1.4)**@{-}?>*{\dir{>}}?>(.5)+(-1.75,0)*{\scriptstyle p * q}} ;
{"6"+C+(1.4,1.4);"2"+C+(-1.4,-1.4)**@{-}?>*{\dir{>}}?>(.5)+(-1.5,0)*{\scriptstyle q}?>(.5)+(1.5,0)*{\scriptstyle \gamma_q^{\bullet}}} ;
{"7"+C+(1.4,1.4);"3"+C+(-1.4,-1.4)**@{-}?>*{\dir{>}}?>(.5)+(-1.5,0)*{\scriptstyle p}?>(.5)+(1.5,0)*{\scriptstyle \gamma_p^{\bullet}}} ;
{"8"+C+(1.4,1.4);"4"+C+(-1.4,-1.4)**@{-}?>*{\dir{>}}} ;
\endxy
\]
We point out that the results we use from \cite{BSM} throughout this proof do not require spaces to be locally compact and Hausdorff. By Proposition 5.2 in \cite{BSM}, we obtain from the above diagram the locally finite BSM $(\gamma_p * \gamma_q)^{\bullet}$ on $p * q : S^{(0)} * T^{(0)} \rightarrow G^{(0)} * G^{(0)},$ where $$(\gamma_p * \gamma_q)^{(u,v)} = \gamma_p^u \times \gamma_q^v.$$

Next, we consider the following pullback diagram in \textbf{\textit{Top}} (this was Diagram D in the proof of Theorem \ref{thm:Haar system for P}). We equip the map $p*q$ with the BSM $(\gamma_p * \gamma_q)^{\bullet}$:
\[ \xymatrix{P^{(0)} \ar [dd]_{\pi_G} \ar [rr]&& S^{(0)} * T^{(0)} \ar [dd]^{(\gamma_p * \gamma_q)^{\bullet}}_{p * q}\\\\
G\ar [rr]^{(r,d)}&&G^{(0)} * G^{(0)}}
\]
We follow Section 4 of \cite{BSM}, where we studied lifting of systems of measures. By Definition 4.1, Remark 4.2 and Proposition 4.4 of \cite{BSM}, we can lift the locally finite BSM $(\gamma_p * \gamma_q)^{\bullet}$ and obtain a locally finite BSM $((r,d)^*(\gamma_p * \gamma_q))^{\bullet}$ on the projection $\pi_G:P^{(0)} \rightarrow G$. We denote $\eta^{\bullet} = ((r,d)^*(\gamma_p * \gamma_q))^{\bullet}$, and from the definition of lifting it follows that for $x \in G$, $\eta^x = \delta_x \times (\gamma_p * \gamma_q)^{(r(x),d(x))} = \delta_x \times \gamma_p^{r(x)} \times \gamma_q^{d(x)}$, which we rewrite as $ \eta^x = \gamma_p^{r(x)} \times \delta_x \times \gamma_q^{d(x)}.$ This completes the proof.
\end{proof}

\begin{lem}\label{lem:eta^x(E)}
Let $E \subseteq P^{(0)}$ be a set of the form $E=(A\times B\times C)\cap P^{(0)}$, where $A\subseteq S^{(0)}$, $B\subseteq G$ and $C\subseteq T^{(0)}$. For any $x \in G$, $$\eta^x (E) = \gamma_p^{r(x)}(A)\delta_x(B)\gamma_q^{d(x)}(C).$$
\end{lem}

\begin{proof}
From the definition of $\eta^{\bullet}$ in Proposition \ref{prop:eta} above, we have that $\eta^x(E) = (\gamma_p^{r(x)} \times \delta_x \times \gamma_q^{d(x)})\left((A\times B\times C)\cap P^{(0)}\right).$ Clearly if $x \notin B$ then $\eta^x(E) = 0$. If $x \in B$ then, since $\delta_x$ is concentrated on $\{x\}$, we can write $\eta^x(E) = (\gamma_p^{r(x)} \times \delta_x \times \gamma_q^{d(x)})\left((A\times \{x\} \times C)\cap P^{(0)}\right).$ A point $(s,x,t) \in P^{(0)}$ whose $G$ component is $x$, satisfies $s \in p^{-1}(r(x))$ and $t \in q^{-1}(d(x))$, hence for $x \in B$ we have $\eta^x(E) = \gamma_p^{r(x)}\left(A \cap p^{-1}(r(x))\right) \cdot \delta_x \left(\{x\}\right) \cdot \gamma_q^{d(x)}\left(C \cap q^{-1}(d(x))\right).$ Since $supp(\gamma_p^{r(x)}) = p^{-1}(r(x))$ and $supp(\gamma_q^{d(x)}) = q^{-1}(d(x))$, it follows that for $x \in B$, $\eta^x(E) = \gamma_p^{r(x)}(A) \delta_x \left(\{x\}\right) \gamma_q^{d(x)}(C).$ We conclude that for any $x \in G$, $\eta^x(E) = \gamma_p^{r(x)}(A) \delta_x (B) \gamma_q^{d(x)}(C).$
\end{proof}

We can now cook up a measure $\mu_P^{(0)}$ on $P^{(0)}$. The ingredients will be the induced measure $\mu_G$ from Definition \ref{def:induced measure muG}, as well as $\eta^{\bullet}$ which we have just constructed.
\begin{mydef}
Let $B \subseteq P^{(0)}$ be a Borel subset. Define:
$$\mu_P^{(0)} (B) := \int_G \eta^x (B) d\mu_G(x).$$
\end{mydef}
\noindent In fact, the measure $\mu_P^{(0)}$ can be written as $$\mu_P^{(0)}=\mu_G\circ[(r,d)^*(\gamma_p*\gamma_q)],$$ as it was obtained by lifting the fibred product of the disintegrations $\gamma_p$ and $\gamma_q$ to $\pi_G:P^{(0)} \rightarrow G$ and then composing with the induced measure of $G$.

In order for $P$ to be a Haar groupoid, $\mu_P^{(0)}$ must be a Radon measure, and in particular locally finite. This is guaranteed modulo our standing Assumption \ref{ass:LB disintegration}.

\begin{prop}\label{prop:muP0 is Radon}
$\mu_P^{(0)}$ is a Radon measure on $P^{(0)}$.
\end{prop}

\begin{proof}
It suffices to show that $\mu_P^{(0)}$ is locally finite. Let $A\subseteq S^{(0)}$, $B\subseteq G$ and $C\subseteq T^{(0)}$ be open subsets with compact closures and consider the set $E=(A\times B\times C)\cap P^{(0)}$, which is an open subset of $P^{(0)}$. Using the definition of $\mu_P^{(0)}$ above along with Lemma \ref{lem:eta^x(E)}, we get $$\mu_P^{(0)}(E)=\int_G \eta^x (E) d\mu_G(x)=\int_G \gamma_p^{r(x)}(A)\delta_x(B)\gamma_q^{d(x)}(C)d\mu_G(x)=\int_B \gamma_p^{r(x)}(A)\gamma_q^{d(x)}(C)d\mu_G(x).$$
It thus follows from Assumption \ref{ass:LB disintegration} that
$$\mu_P^{(0)}(E)\leq \left(\sup_{s}\gamma^s_p(A)\right)\cdot\left(\sup_{t}\gamma^t_q(C)\right)\cdot\mu_G(B)<\infty.$$
Since the open sets of the same form as $E$ constitute a basis for the topology of $P^{(0)}$, we conclude that $\mu_P^{(0)}$ is locally finite.
\end{proof}
Note that an alternative proof of Proposition \ref{prop:muP0 is Radon} is obtained by arguing that the system $\eta^{\bullet}$ is locally bounded (modulo Assumption \ref{ass:LB disintegration}), and then applying Corollary 3.7 of \cite{BSM}.

\begin{prop}
The measure $\mu_P^{(0)}$ is independent of the choice of the disintegrations $\gamma_p^{\bullet}$ and $\gamma_q^{\bullet}$.
\end{prop}

\begin{proof}
Let $\widetilde{\gamma_p}^{\bullet}$ and $\widetilde{\gamma_q}^{\bullet}$ be two other disintegrations on $p$ and $q$ respectively, and let $\widetilde{\mu}_P^{(0)}$ be the corresponding measure on $P^{(0)}$. By Corollary 6.6 in \cite{BSM}, $\widetilde{\gamma_p}^{u}=\gamma_p^{u}$  and $\widetilde{\gamma_q}^{u}=\gamma_q^{u}$ for $\mu_G^{(0)}$-almost every $u$ in $G^{(0)}$.

Let $A\subseteq S^{(0)}$, $B\subseteq G$ and $C\subseteq T^{(0)}$ be open and let $E=(A\times B\times C)\cap P^{(0)}$ be the corresponding open subset of $P^{(0)}$. By the calculation in the proof of Proposition \ref{prop:muP0 is Radon} above, $\mu_P^{(0)}(E)=\int_B \gamma_p^{r(x)}(A)\gamma_q^{d(x)}(C)d\mu_G(x),$ and likewise $\widetilde{\mu}_P^{(0)}(E)=\int_B \widetilde{\gamma_p}^{r(x)}(A)\widetilde{\gamma_q}^{d(x)}(C)d\mu_G(x).$ Using Lemma \ref{lem:integrating muG vs muG0} and the fact that $supp(\lambda_G^u) = r^{-1}(u)$ we get
\begin{eqnarray*}
\widetilde{\mu}_P^{(0)}(E) &=& \int_B \widetilde{\gamma_p}^{r(x)}(A)\widetilde{\gamma_q}^{d(x)}(C)d\mu_G(x) \ = \ \int_{G^{(0)}} \left(\int_B \widetilde{\gamma_p}^{r(x)}(A)\widetilde{\gamma_q}^{d(x)}(C) d\lambda_G^u(x) \right) d\mu_G^{(0)}(u) \\
&=& \int_{G^{(0)}} \widetilde{\gamma_p}^u(A) \left(\int_B \widetilde{\gamma_q}^{d(x)}(C) d\lambda_G^u(x) \right) d\mu_G^{(0)}(u) \\
&=& \int_{G^{(0)}} \gamma_p^u(A) \left(\int_B \widetilde{\gamma_q}^{d(x)}(C) d\lambda_G^u(x) \right) d\mu_G^{(0)}(u) \\
&=& \int_{G^{(0)}} \left(\int_B \gamma_p^{r(x)}(A) \widetilde{\gamma_q}^{d(x)}(C) d\lambda_G^u(x) \right) d\mu_G^{(0)}(u)
\end{eqnarray*}
Justification for the next step is based on formula (\ref{eq:Delta inverse}) of Remark \ref{rem:Delta}. The remaining calculation retraces the previous arguments. \begin{eqnarray*}
&=& \int_{G^{(0)}}\left(\int_B \gamma_p^{d(x)}(A)\widetilde{\gamma_q}^{r(x)}(C)\Delta_G^{-1}(x)d\lambda_G^u(x)\right)d\mu_G^{(0)}(u) \\
&=& \int_{G^{(0)}}\widetilde{\gamma_q}^{u}(C)\left(\int_B \gamma_p^{d(x)}(A)\Delta_G^{-1}(x)d\lambda_G^u(x)\right)d\mu_G^{(0)}(u) \\
&=& \int_{G^{(0)}}\gamma_q^{u}(C)\left(\int_B \gamma_p^{d(x)}(A)\Delta_G^{-1}(x)d\lambda_G^u(x)\right)d\mu_G^{(0)}(u) \\
&=& \int_{G^{(0)}}\left(\int_B \gamma_p^{d(x)}(A)\gamma_q^{r(x)}(C)\Delta_G^{-1}(x)d\lambda_G^u(x)\right)d\mu_G^{(0)}(u) \\
&=& \int_{G^{(0)}}\left(\int_B \gamma_p^{r(x)}(A)\gamma_q^{d(x)}(C)d\lambda_G^u(x)\right)d\mu_G^{(0)}(u) \ = \ \mu_P^{(0)}(E)
\end{eqnarray*}
Thus, $\widetilde{\mu}_P^{(0)}(E) = \mu_P^{(0)}(E)$ for any open set of the form $E=(A\times B\times C)\cap P^{(0)}$. These sets constitute a countable basis $\mathcal{B}^{(0)}$ for the topology of $P^{(0)}$, in analogy to Remark \ref{rem:elementary open subsets of P}. Therefore, since $\mu_P^{(0)}$ is locally finite, it follows that $\widetilde{\mu}_P^{(0)}$ is locally finite as well. Moreover, $\mu_P^{(0)}$ and $\widetilde{\mu}_P^{(0)}$ agree on finite intersections of sets in $\mathcal{B}^{(0)}$ as these sets are also in $\mathcal{B}^{(0)}$, so we can now use Lemma 2.24 of \cite{BSM}, as in the proof of Proposition \ref{prop:left invariance of lambda P}, and conclude that $\widetilde{\mu}_P^{(0)} = \mu_P^{(0)}$.
\end{proof}

\noindent The following is a simple observation, whose proof is analogous to the proof of Lemma \ref{lem:integrating muG vs muG0}, and thus omitted.
\begin{lem}\label{lem:integrating dmuP0}
For any Borel function $f$ on $P^{(0)}$:
$$\int_{P^{(0)}} f(u) d\mu_P^{(0)}(u) = \int_{G} \left( \int_{P^{(0)}} f(u) d\eta^y (u) \right) d\mu_G (y).$$
\end{lem}

In \S 3 of \cite{BSM} we defined the composition $(\beta\circ\alpha)^{\bullet}$ of BSMs $\xymatrix{X\ar [rr]^{p}_{\alpha^{\bullet}}&&Y\ar [rr]^{q}_{\beta^{\bullet}}&&Z}$, which is characterized by
\begin{equation}\label{eq:characterization of composition}
\int_X f(x)d (\beta \circ \alpha)^z (x) = \int_Y \left(\int_X f(x)d\alpha^y (x) \right) d\beta^z(y).
\end{equation}
This will be essential for proving the following lemma.
\begin{lem}\label{lem:commuting triple integrals}
For any Borel function $f(y,\sigma)$ on $G*S$,
$$ \int_{\!\!S^{(0)}\!\!} \int_S \int_G f(y,\sigma) d\lambda_G^{p(r_S(\sigma))}(y) d\lambda_S^s(\sigma) d\gamma_p^u(s) = \int_G \int_{\!\!S^{(0)}\!\!} \int_S f(y,\sigma) d\lambda_S^s(\sigma) d\gamma_p^{r_G(y)}(s) d\lambda_G^u(y). $$
\end{lem}

\begin{proof}
Consider the composition $(\gamma_p \circ \lambda_S)^{\bullet}$ of the BSMs $\xymatrix{S\ar [rr]^{r_S}_{\lambda_S^{\bullet}}&&S^{(0)}\ar [rr]^{p}_{\gamma_p^{\bullet}}&&G^{(0)}}$. We use this as the right edge in the pull-back diagram below. Following \S 4 of \cite{BSM}, we lift the BSM $\lambda_G^{\bullet}$ to obtain a BSM $((p \circ r_S)^*\lambda_G)^{\bullet}$ on $\pi_S:G*S \rightarrow S$, and
we lift the BSM $(\gamma_p \circ \lambda_S)^{\bullet}$ to obtain a BSM $(r_G^*(\gamma_p \circ \lambda_S))^{\bullet}$ on $\pi_G:G*S \rightarrow G$.
\[
\begin{array}{cc}
\xymatrix{G*S \ar [dd]_{\pi_G}^{(r_G^*(\gamma_p \circ \lambda_S))^{\bullet}} \ar [rr]^{\pi_S}_{((p \circ r_S)^*\lambda_G)^{\bullet}} && S \ar [dd]_{p \circ r_S}^{(\gamma_p \circ \lambda_S)^{\bullet}} \\\\
G\ar [rr]^{r_G}_{\lambda_G^{\bullet}} && G^{(0)}}
\end{array}
\]
By the definition of lifting, $$((p \circ r_S)^*\lambda_G)^{\sigma} = \lambda_G^{p(r_S(\sigma))} \times \delta_{\sigma}, \qquad \sigma \in S$$ and
$$({r_G}^*(\gamma_p \circ \lambda_S))^{y} = \delta_y \times (\gamma_p \circ \lambda_S)^{r_G(y)}, \qquad y \in G.$$
The above diagram gives rise to two compositions: $\xymatrix{G*S\ar [rr]^{\pi_S}_{((p \circ r_S)^*\lambda_G)^{\bullet}}&&S \ar [rr]^{p \circ r_S}_{(\gamma_p \circ \lambda_S)^{\bullet}}&&G^{(0)}}$ and $\xymatrix{G*S\ar [rr]^{\pi_G}_{(r_G^*(\gamma_p \circ \lambda_S))^{\bullet}}&&G \ar [rr]^{r_G}_{\lambda_G^{\bullet}}&&G^{(0)}}$. However, proposition 4.8 of \cite{BSM} states that the above diagram is a commutative diagram of BSMs, and explicitly,
$$[(\gamma_p \circ \lambda_S) \circ ((p \circ r_S)^*\lambda_G)]^{\bullet} = [\lambda_G \circ ({r_G}^*(\gamma_p \circ \lambda_S)]^{\bullet},$$ as BSMs on $G*S \rightarrow G^{(0)}$. The above equality implies that for any Borel function $f(y,\sigma)$ on $G*S$,
\[
\int_{G*S} f(y,\sigma) d((\gamma_p \circ \lambda_S) \circ ((p \circ r_S)^*\lambda_G))^{u}(y,\sigma) = \int_{G*S} f(y,\sigma) d(\lambda_G \circ ({r_G}^*(\gamma_p \circ \lambda_S))^{u}(y,\sigma).
\]
We expand the left and the right hand sides of the above equality separately, using repeatedly the characterization (\ref{eq:characterization of composition}) of composition of BSMs above:
\begin{eqnarray*}
LHS &=& \int_{G*S} f(y,\sigma) d((\gamma_p \circ \lambda_S) \circ ((p \circ r_S)^*\lambda_G))^{u}(y,\sigma) \\
&=& \int_S \left(  \int_{G*S} f(y,\sigma) d((p \circ r_S)^*\lambda_G)^{\widetilde{\sigma}}(y,\sigma) \right) d(\gamma_p \circ \lambda_S)^u (\widetilde{\sigma}) \\
&=& \int_{S^{(0)}} \int_S \left(  \int_{G*S} f(y,\sigma) d((p \circ r_S)^*\lambda_G)^{\widetilde{\sigma}}(y,\sigma) \right) d\lambda_S^s(\widetilde{\sigma}) d\gamma_p^u(s) \\
&=& \int_{S^{(0)}} \int_S \left(  \int_{G*S} f(y,\sigma) d(\lambda_G^{p(r_S(\widetilde{\sigma}))} \times \delta_{\widetilde{\sigma}})(y,\sigma) \right) d\lambda_S^s(\widetilde{\sigma}) d\gamma_p^u(s) \\
&=& \int_{S^{(0)}} \int_S \left(  \int_{G*S} f(y,\sigma) d\lambda_G^{p(r_S(\widetilde{\sigma}))}(y) d \delta_{\widetilde{\sigma}}(\sigma) \right) d\lambda_S^s(\widetilde{\sigma}) d\gamma_p^u(s) \\
&=& \int_{S^{(0)}} \int_S \int_G f(y,\sigma) d\lambda_G^{p(r_S(\sigma))}(y) d\lambda_S^s(\sigma) d\gamma_p^u(s)
\end{eqnarray*}

\begin{eqnarray*}
RHS &=& \int_{G*S} f(y,\sigma) d(\lambda_G \circ ({r_G}^*(\gamma_p \circ \lambda_S))^{u}(y,\sigma) \\
&=& \int_G \left( \int_{G*S} f(y,\sigma) d({r_G}^*(\gamma_p \circ \lambda_S))^{\widetilde{y}}(y,\sigma) \right) d\lambda_G^u(\widetilde{y}) \\
&=& \int_G \left( \int_{G*S} f(y,\sigma) d(\delta_{\widetilde{y}} \times (\gamma_p \circ \lambda_S)^{r_G({\widetilde{y}})})(y,\sigma) \right) d\lambda_G^u(\widetilde{y}) \\
&=& \int_G \left( \int_{G*S} f(y,\sigma) d\delta_{\widetilde{y}}(y) d(\gamma_p \circ \lambda_S)^{r_G({\widetilde{y}})}(\sigma) \right) d\lambda_G^u(\widetilde{y}) \\
&=& \int_G \left( \int_S f(y,\sigma) d(\gamma_p \circ \lambda_S)^{r_G(y)}(\sigma) \right) d\lambda_G^u(y) \\
&=& \int_G \int_{S^{(0)}} \int_S f(y,\sigma) d\lambda_S^s(\sigma) d\gamma_p^{r_G(y)}(s) d\lambda_G^u(y)
\end{eqnarray*}
Since the above expressions are equal, this yields the desired formula.
\end{proof}

\begin{lem}\label{lem:expanding integral muP}
Let $f(\sigma,x,\tau)$ be a Borel function on $P$. Then $$\hspace{-10cm} \int_P f(\sigma,x,\tau) d\mu_P(\sigma,x,\tau) = $$
$$\int_{G^{(0)}}\int_G\int_{S^{(0)}}\int_{S}\int_{T^{(0)}}\int_{T} f(\sigma,y,\tau)d\lambda_T^{t}(\tau)d\gamma_q^{d(y)}(t) d\lambda_S^{s}(\sigma) d\gamma_p^{r(y)}(s)d\lambda^u_G(y)d\mu_G^{(0)}(u).$$
\end{lem}

\begin{proof}
$$\hspace{-10cm} \int_P f(\sigma,x,\tau) d\mu_P(\sigma,x,\tau) = $$
\begin{eqnarray*}
&=& \int_{P^{(0)}} \int_P f(\sigma,x,\tau) d\lambda_P^{(s,g,t)}(\sigma,x,\tau) d\mu_P^{(0)}(s,g,t) \hspace{3cm} (\text{by Lemma \ref{lem:integrating muG vs muG0}}) \\
&=& \int_G\int_{P^{(0)}}\int_P f(\sigma,x,\tau)d\lambda_P^{(s,g,t)}(\sigma,x,\tau) d\eta^y(s,g,t) d\mu_G(y)
\hspace{1.5cm} (\text{by Lemma \ref{lem:integrating dmuP0}})
\end{eqnarray*}
Rewriting $\eta^y$ by Proposition \ref{prop:eta}, and then rewriting $\lambda_P^{(s,g,t)}$ by Definition \ref{def:Haar for P}, we get
\begin{eqnarray*}
&\!=\!& \int_G\int\!\!\!\!\int\!\!\!\!\int_{S^{(0)}\times G\times T^{(0)}}\int_P f(\sigma,x,\tau) d\lambda_P^{(s,g,t)}(\sigma,x,\tau) d\gamma_p^{r(y)}(s)d\delta_y(g)d\gamma_q^{d(y)}(t)d\mu_G(y) \\
&\!=\!& \int_G\int\!\!\!\!\int\!\!\!\!\int_{S^{(0)}\times G\times T^{(0)}}\int\!\!\!\!\int\!\!\!\!\int_{S\times G\times T} f(\sigma,x,\tau)d\lambda_S^{s}(\sigma)d\delta_g(x)d\lambda_T^{t}(\tau) d\gamma_p^{r(y)}(s)d\delta_y(g)d\gamma_q^{d(y)}(t)d\mu_G(y) \\
&\!=\!& \int_G\int\!\!\!\!\int_{S^{(0)}\times T^{(0)}}\int\!\!\!\!\int_{S\times T} f(\sigma,y,\tau)d\lambda_S^{s}(\sigma)d\lambda_T^{t}(\tau)  d\gamma_p^{r(y)}(s)d\gamma_q^{d(y)}(t)d\mu_G(y)
\end{eqnarray*}
Using Lemma \ref{lem:integrating muG vs muG0} again, followed by Fubini's theorem, we have
\begin{eqnarray*}
&=& \int_{G^{(0)}}\int_G\int\!\!\!\!\int_{S^{(0)}\times T^{(0)}}\int\!\!\!\!\int_{S\times T} f(\sigma,y,\tau)d\lambda_S^{s}(\sigma)d\lambda_T^{t}(\tau)  d\gamma_p^{r(y)}(s)d\gamma_q^{d(y)}(t)d\lambda^u_G(y)d\mu_G^{(0)}(u) \\
&=& \int_{G^{(0)}}\int_G\int_{S^{(0)}}\int_{T^{(0)}}\int_{S}\int_{T} f(\sigma,y,\tau)d\lambda_T^{t}(\tau)d\lambda_S^{s}(\sigma)  d\gamma_q^{d(y)}(t)d\gamma_p^{r(y)}(s)d\lambda^u_G(y)d\mu_G^{(0)}(u)
\end{eqnarray*}
We now invoke Proposition 5.6 from \cite{BSM}, which asserts that for locally finite BSMs, fibred products commute with compositions. We apply this theorem to the following diagram (it is straightforward to verify that the conditions for the proposition indeed hold. In particular, $\lambda_S^{\bullet}$ and $\lambda_T^{\bullet}$ are locally bounded). We obtain that $(\gamma_q *\gamma_p) \circ (\lambda_T * \lambda_S) = (\gamma_q \circ \lambda_T)*(\gamma_p \circ \lambda_S)$.

$$
\xy \xymatrix@R=.2in@C=.2in{ & T*S\ar[rr]^{r_T *r_S}_{(\lambda_T *\lambda_S)^{\bullet}}\ar@{-}[dd]\ar[ddl]&&T^{(0)}*S^{(0)}\ar[rr]^{q*p}_{(\gamma_q *\gamma_p)^{\bullet}}\ar@{-}[dd]\ar[ddl]&&G^{(0)}*G^{(0)}\ar[dddd]\ar[ddl]\\%
                          \hspace{.75in}&&\hspace{.5in}&&&\hspace{.75in}\\%
              S\ar[rr]^(.65){r_S}_(.65){\lambda_S^{\bullet}}\ar[dddd]   &{\ }^{\ }\ar[dd]&S^{(0)}\ar[rr]^(.65){p}_(.65){\gamma_p^{\bullet}}\ar[dddd]&{\ }^{\ } \ar[dd]&G^{(0)}\ar[dddd]&\vspace{-2in}\\%
              &&&&&\\%
                          &T\ar@{-}[r]^(.65){r_T}_(.65){\lambda_T^{\bullet}}  \ar[ddl] &\hspace{.06in}\ar[r]&T^{(0)}\ar@{-}[r]^(.65){q}_(.65){\gamma_q^{\bullet}}\ar[ddl]&\hspace{.06in}\ar[r]&G^{(0)}\ar[ddl]\\%
                          &&&&&\\%
             \underline{G}\ar[rr]^{\scriptstyle\rm id}&&\underline{G}\ar[rr]^{\scriptstyle\rm id}&&\underline{G}&} \endxy%
$$

\noindent Therefore, returning to our main calculation, we get
\begin{eqnarray*}
&=& \int_{G^{(0)}}\int_G\int_{S^{(0)}}\int_{S}\int_{T^{(0)}}\int_{T} f(\sigma,y,\tau)d\lambda_T^{t}(\tau)d\gamma_q^{d(y)}(t) d\lambda_S^{s}(\sigma) d\gamma_p^{r(y)}(s)d\lambda^u_G(y)d\mu_G^{(0)}(u)
\end{eqnarray*}
This completes the proof.
\end{proof}

\begin{prop}\label{prop:muP is quasi invariant}
The measure $\mu_P^{(0)}$ is quasi-invariant with respect to $\lambda_P^{\bullet}$.
\end{prop}

\begin{proof}
By definition \ref{def:quasi-invariant measure}, we need to show that $\mu_P$ and $\mu_P^{-1}$ are mutually absolutely continuous. We recall from Definition \ref{def:induced measure muG} that $\mu_P$ is the induced measure, defined for any Borel set $E \subseteq P$ by $\mu_P (E) = \int_{P^{(0)}} \lambda_P^v (E) d\mu_P^{(0)} (v)$, and $\mu_P^{-1}$ is its image under inversion, i.e. $\mu_P^{-1}(E) = \mu_P(E^{-1})$. We will prove:

\noindent \underline{Claim}: There exists a function $\Lambda:P \rightarrow \mathbb{R}$ satisfying $\Lambda(\alpha) >0$ $\mu_P$-a.e., such that for any Borel set $E \subseteq P$, $\displaystyle \mu_P^{-1}(E)=\int_P \chi_{_E}(\alpha) \Lambda(\alpha) d\mu_P(\alpha).$

\noindent It will then follow that $\mu_P \sim \mu_P^{-1}$, since $\mu_P(E)=\int_P \chi_{_E}(\alpha) d\mu_P(\alpha)$. In fact, $\Delta = \Lambda^{-1}$ will be the modular function of $\mu_P$.

We first prove the claim for elementary open subsets of the form $E=(A\times B\times C)\cap P$, where $A\subseteq S$, $B\subseteq G$ and $C\subseteq T$. Note that the characteristic function $\chi_{_E}$ is the restriction of the product $\chi_{_A}\cdot\chi_{_B}\cdot\chi_{_C}$ to $P$. \\

\noindent We denote $\alpha = (\sigma,x,\tau) \in P$ and $v = (s,g,t) \in P^{(0)}$. By Lemma \ref{lem:expanding integral muP}:
$$\hspace{-7cm} \mu_P^{-1}(E) \ = \ \mu_P(E^{-1}) \ = \ \int_P \chi_{_{E^{-1}}}(\sigma,x,\tau) d\mu_P(\sigma,x,\tau) $$
\begin{eqnarray*}
&=& \int_{G^{(0)}}\int_G\int_{S^{(0)}}\int_{S}\int_{T^{(0)}}\int_{T} \chi_{_{E^{-1}}}(\sigma,y,\tau)d\lambda_T^{t}(\tau)d\gamma_q^{d(y)}(t) \\
& & \hspace{8.3cm} d\lambda_S^{s}(\sigma) d\gamma_p^{r(y)}(s)d\lambda^u_G(y)d\mu_G^{(0)}(u) \\
&=& \int_{G^{(0)}}\int_G\int_{S^{(0)}}\int_{S}\int_{T^{(0)}}\int_{T} \chi_{_E}(\sigma^{-1},p(\sigma)^{-1} y q(\tau),\tau^{-1})d\lambda_T^{t}(\tau)d\gamma_q^{d(y)}(t) \\
& & \hspace{8.3cm} d\lambda_S^{s}(\sigma) d\gamma_p^{r(y)}(s)d\lambda^u_G(y)d\mu_G^{(0)}(u) \\
&=& \int_{G^{(0)}}\int_G\int_{S^{(0)}}\int_{S}\int_{T^{(0)}}\int_{T} \chi_{_A}(\sigma^{-1})\chi_{_B}(p(\sigma)^{-1}y q(\tau))\chi_{_C}(\tau^{-1})d\lambda_T^{t}(\tau)d\gamma_q^{d(y)}(t) \\
& & \hspace{8.3cm} d\lambda_S^{s}(\sigma) d\gamma_p^{r(y)}(s)d\lambda^u_G(y)d\mu_G^{(0)}(u)
\end{eqnarray*}
Using Lemma \ref{lem:commuting triple integrals}, we obtain
\begin{eqnarray*}
&=& \int_{G^{(0)}}\int_{S^{(0)}}\int_{S}\int_G\int_{T^{(0)}}\int_{T} \chi_{_A}(\sigma^{-1}) \chi_{_B}(p(\sigma)^{-1}yq(\tau))\chi_{_C}(\tau^{-1})d\lambda_T^{t}(\tau)d\gamma_q^{d(y)}(t) \\
& & \hspace{8.3cm} d\lambda^{p(r(\sigma))}_G(y)d\lambda_S^{s}(\sigma) d\gamma_p^{u}(s)d\mu_G^{(0)}(u) \\
&=& \int_{G^{(0)}}\int_{S^{(0)}}\int_{S}\chi_{_A}(\sigma^{-1})\int_G\int_{T^{(0)}}\int_{T} \chi_{_B}(p(\sigma)^{-1}yq(\tau))\chi_{_C}(\tau^{-1})d\lambda_T^{t}(\tau)d\gamma_q^{d(y)}(t) \\
& & \hspace{8.3cm} d\lambda^{p(r(\sigma))}_G(y)d\lambda_S^{s}(\sigma) d\gamma_p^{u}(s)d\mu_G^{(0)}(u)
\end{eqnarray*}
Let $f_1$ be a function on $G$, defined by the formula
$$f_1(y)=\int_{T^{(0)}}\int_{T} \chi_{_B}(yq(\tau))\chi_{_C}(\tau^{-1})d\lambda_T^{t}(\tau)d\gamma_q^{d(y)}(t).$$
From Lemma 7.3 of \cite{BSM} we know that a system of measures $\lambda^{\bullet}$ on a groupoid $G$ is left invariant if and only if for any $x \in G$ and every non-negative Borel function $f$ on $G$,
\begin{equation}\label{eq:left invariance}
\int f(xy)d\lambda^{d(x)}(y) = \int f(y)d\lambda^{r(x)}(y).
\end{equation}
This implies, using $x = p(\sigma)^{-1}$ and the above $f_1$, that
$$\int_G f_1(p(\sigma)^{-1}y)d\lambda^{p(r(\sigma))}_G=\int_G f_1(y)d\lambda^{p(d(\sigma))}_G.$$ Therefore, returning to our main calculation and noting that $d(p(\sigma)^{-1}y) = d(y)$, we have
\begin{eqnarray*}
\mu_P^{-1}(E) &=& \int_{G^{(0)}}\int_{S^{(0)}}\int_{S}\chi_{A}(\sigma^{-1})\int_G\int_{T^{(0)}}\int_{T} \chi_{_B}(y q(\tau))\chi_{_C}(\tau^{-1})d\lambda_T^{t}(\tau)d\gamma_q^{d(y)}(t) \\
& & \hspace{8.3cm} d\lambda^{p(d(\sigma))}_G(y)d\lambda_S^{s}(\sigma) d\gamma_p^{u}(s)d\mu_G^{(0)}(u)
\end{eqnarray*}
Using the fact that $\gamma_p^{\bullet}$ is a disintegration of $\mu_S^{(0)}$ with respect to $\mu_G^{(0)}$, followed by Lemma \ref{lem:integrating muG vs muG0}, we get
\begin{eqnarray*}
&=& \int_{S^{(0)}}\int_{S}\chi_{A}(\sigma^{-1})\int_G\int_{T^{(0)}}\int_{T} \chi_{_B}(y q(\tau))\chi_{_C}(\tau^{-1})d\lambda_T^{t}(\tau)d\gamma_q^{d(y)}(t) d\lambda^{p(d(\sigma))}_G(y) d\lambda_S^{s}(\sigma) d\mu_S^{(0)}(s) \\
&=& \int_{S}\chi_{A}(\sigma^{-1})\int_G\int_{T^{(0)}}\int_{T} \chi_{_B}(y q(\tau))\chi_{_C}(\tau^{-1})d\lambda_T^{t}(\tau)d\gamma_q^{d(y)}(t) d\lambda^{p(d(\sigma))}_G(y) d\mu_S(\sigma) \\
&=& \int_{S}\int_G\int_{T^{(0)}}\int_{T} \chi_{A}(\sigma^{-1})\chi_{_B}(y q(\tau))\chi_{_C}(\tau^{-1})d\lambda_T^{t}(\tau)d\gamma_q^{d(y)}(t) d\lambda^{p(d(\sigma))}_G(y) d\mu_S(\sigma)
\end{eqnarray*}
The measure $\mu_S^{(0)}$ is quasi-invariant. Therefore, formula (\ref{eq:Delta inverse}) of Remark \ref{rem:Delta} permits us to replace $\sigma^{-1}$ by $\sigma$ at the price of inserting $\Delta_S^{-1}(\sigma)$:
\begin{eqnarray*}
&=& \int_{S} \int_G\int_{T^{(0)}}\int_{T} \chi_{A}(\sigma) \chi_{_B}(y q(\tau))\chi_{_C}(\tau^{-1})\Delta_S^{-1}(\sigma)d\lambda_T^{t}(\tau)d\gamma_q^{d(y)}(t) d\lambda^{p(r(\sigma))}_G(y) d\mu_S(\sigma)
\end{eqnarray*}
Re-expanding $d\mu_S$ and then using Lemma \ref{lem:commuting triple integrals} again, followed by Lemma \ref{lem:integrating muG vs muG0}, we have
\begin{eqnarray*}
&=& \int_{G^{(0)}}\int_{S^{(0)}}\int_{S} \int_G\int_{T^{(0)}}\int_{T} \chi_{A}(\sigma) \chi_{_B}(y q(\tau))\chi_{_C}(\tau^{-1}) \Delta_S^{-1}(\sigma) d\lambda_T^{t}(\tau)d\gamma_q^{d(y)}(t) \\
& & \hspace{8.3cm} d\lambda^{p(r(\sigma))}_G(y) d\lambda_S^{s}(\sigma) d\gamma_p^{u}(s)d\mu_G^{(0)}(u) \\
&=& \int_{G^{(0)}}\int_G\int_{S^{(0)}}\int_{S}\int_{T^{(0)}}\int_{T} \chi_{_A}(\sigma)\chi_{_B}(y q(\tau))\chi_{_C}(\tau^{-1})\Delta_S^{-1}(\sigma)d\lambda_T^{t}(\tau)d\gamma_q^{d(y)}(t) \\
& & \hspace{9cm} d\lambda_S^{s}(\sigma) d\gamma_p^{r(y)}(s)d\lambda^{u}_G(y)d\mu_G^{(0)}(u) \\
&=& \int_G\int_{S^{(0)}}\int_{S}\int_{T^{(0)}}\int_{T} \chi_{_A}(\sigma)\chi_{_B}(y q(\tau))\chi_{_C}(\tau^{-1})\Delta_S^{-1}(\sigma)d\lambda_T^{t}(\tau)d\gamma_q^{d(y)}(t) \\
& & \hspace{10cm} d\lambda_S^{s}(\sigma) d\gamma_p^{r(y)}(s)d\mu_G(y)
\end{eqnarray*}
We now use the quasi-invariance of $\mu_G^{(0)}$ and formula (\ref{eq:Delta inverse}) of Remark \ref{rem:Delta} to write
\begin{eqnarray*}
&=& \int_G\int_{S^{(0)}}\int_{S}\int_{T^{(0)}}\int_{T}\chi_{_A}(\sigma) \chi_{_B}(y^{-1} q(\tau))\chi_{_C}(\tau^{-1})\Delta_S^{-1}(\sigma)\Delta_G^{-1}(y)d\lambda_T^{t}(\tau)d\gamma_q^{r(y)}(t) \\
& & \hspace{10cm} d\lambda_S^{s}(\sigma) d\gamma_p^{d(y)}(s)d\mu_G(y)
\end{eqnarray*}
Next, we apply the characterization (\ref{eq:characterization of composition}) preceding Lemma \ref{lem:commuting triple integrals} above to the compositions $\xymatrix{S\ar [rr]^{r_S}_{\lambda_S^{\bullet}}&&S^{(0)}\ar [rr]^{p}_{\gamma_p^{\bullet}}&&G^{(0)}}$ and $\xymatrix{T\ar [rr]^{r_T}_{\lambda_T^{\bullet}}&&T^{(0)}\ar [rr]^{q}_{\gamma_q^{\bullet}}&&G^{(0)}}$. We obtain
\begin{eqnarray*}
&=& \int_G\int_{S}\int_{T}\chi_{_A}(\sigma) \chi_{_B}(y^{-1} q(\tau))
\chi_{_C}(\tau^{-1})\Delta_S^{-1}(\sigma)\Delta_G^{-1}(y)d(\gamma_q \circ \lambda_T)^{r(y)}(\tau) \\
& & \hspace{10.8cm} d(\gamma_p \circ \lambda_S)^{d(y)}(\sigma)d\mu_G(y)
\end{eqnarray*}
We can now use Fubini's theorem, after which we re-expand the compositions as well as $\mu_G$:
\begin{eqnarray*}
&=& \int_G\int_{T}\int_{S}\chi_{_A}(\sigma)\chi_{_B}(y^{-1} q(\tau))\chi_{_C}(\tau^{-1})\Delta_S^{-1}(\sigma)\Delta_G^{-1}(y)d(\gamma_p \circ \lambda_S)^{d(y)}(\sigma) \\
& & \hspace{10.8cm} d(\gamma_q \circ \lambda_T)^{r(y)}(\tau) d\mu_G(y) \\
&=& \int_{G^{(0)}}\int_G\int_{T^{(0)}}\int_{T}\int_{S^{(0)}}\int_{S}\chi_{_A}(\sigma)\chi_{_B}(y^{-1} q(\tau))\chi_{_C}(\tau^{-1})\Delta_S^{-1}(\sigma)\Delta_G^{-1}(y)d\lambda_S^{s}(\sigma) d\gamma_p^{d(y)}(s) \\
& & \hspace{9.5cm} d\lambda_T^{t}(\tau)d\gamma_q^{r(y)}(t)d\lambda^{u}_G(y)d\mu_G^{(0)}(u)
\end{eqnarray*}
By Lemma \ref{lem:commuting triple integrals} with $T,t,\tau$ and $q$ in place of $S,s,\sigma$ and $p$, we get
\begin{eqnarray*}
&=& \int_{G^{(0)}}\int_{T^{(0)}}\int_{T}\int_G\int_{S^{(0)}}\int_{S}\chi_{_A}(\sigma)\chi_{_B}(y^{-1} q(\tau))\chi_{_C}(\tau^{-1})\Delta_S^{-1}(\sigma)\Delta_G^{-1}(y)d\lambda_S^{s}(\sigma) d\gamma_p^{d(y)}(s) \\
& & \hspace{9.3cm} d\lambda_G^{q(r(\tau))}(y)d\lambda_T^{t}(\tau)d\gamma_q^{u}(t)d\mu_G^{(0)}(u) \\
&=& \int_{G^{(0)}}\int_{T^{(0)}}\int_{T}\int_G\int_{S^{(0)}}\int_{S}\chi_{_A}(\sigma)\chi_{_B^{-1}} (q(\tau)^{-1}y)\chi_{_C}(\tau^{-1})\Delta_S^{-1}(\sigma)\Delta_G^{-1}(y)d\lambda_S^{s}(\sigma) d\gamma_p^{d(y)}(s) \\
& & \hspace{9.3cm} d\lambda_G^{q(r(\tau))}(y)d\lambda_T^{t}(\tau)d\gamma_q^{u}(t)d\mu_G^{(0)}(u)
\end{eqnarray*}
Let $f_2$ be a function on $G$, defined by the formula
$$f_2(y)=\int_{S^{(0)}}\int_{S}\chi_{_A}(\sigma)\chi_{_B^{-1}} (y)\chi_{_C}(\tau^{-1})\Delta_S^{-1}(\sigma)\Delta_G^{-1}(q(\tau))\Delta_G^{-1}(y)d\lambda_S^{s}(\sigma) d\gamma_p^{d(y)}(s).$$ Using $x = q(\tau)^{-1}$ and $f_2$ in Equation (\ref{eq:left invariance}) above, we obtain that
$$\int_G f_2(q(\tau)^{-1}y)d\lambda^{q(r(\tau))}_G=\int_G f_2(y)d\lambda^{q(d(\tau))}_G.$$ Recall that we take $\Delta_G$ to be a groupoid homomorphism (see Remark \ref{rem:Delta}). Therefore, $\Delta_G^{-1}(q(\tau))\Delta_G^{-1}(q(\tau)^{-1}y) = \Delta_G^{-1}(q(\tau))\Delta_G^{-1}(q(\tau)^{-1})\Delta_G^{-1}(y) = \Delta_G^{-1}(y)$. Hence, noting also that $d(q(\tau)^{-1}y) = d(y)$, the left hand side of the above equality gives precisely the last line of our main calculation. From the right hand side we then get
\begin{eqnarray*}
\mu_P^{-1}(E) &=& \int_{G^{(0)}}\int_{T^{(0)}}\int_{T}\int_G\int_{S^{(0)}}\int_{S}\chi_{_A}(\sigma)\chi_{_{B^{-1}}}(y)\chi_{_C}(\tau^{-1})
\Delta_S^{-1}(\sigma)\Delta_G^{-1}(q(\tau))\Delta_G^{-1}(y) \\
& & \hspace{5cm} d\lambda_S^{s}(\sigma) d\gamma_p^{d(y)}(s) d\lambda_G^{q(d(\tau))}(y)d\lambda_T^{t}(\tau)d\gamma_q^{u}(t)d\mu_G^{(0)}(u)
\end{eqnarray*}
From the fact that $\gamma_q^{\bullet}$ is a disintegration of $\mu_T^{(0)}$ with respect to $\mu_G^{(0)}$, followed by Lemma \ref{lem:integrating muG vs muG0}, we get
\begin{eqnarray*}
&=& \int_{T^{(0)}}\int_{T}\int_G\int_{S^{(0)}}\int_{S}\chi_{_A}(\sigma)\chi_{_{B^{-1}}}(y)\chi_{_C}(\tau^{-1})
\Delta_S^{-1}(\sigma)\Delta_G^{-1}(q(\tau))\Delta_G^{-1}(y) \\
& & \hspace{8cm} d\lambda_S^{s}(\sigma) d\gamma_p^{d(y)}(s) d\lambda_G^{q(d(\tau))}(y)d\lambda_T^{t}(\tau)d\mu_T^{(0)}(t) \\
&=& \int_{T}\int_G\int_{S^{(0)}}\int_{S}\chi_{_A}(\sigma)\chi_{_{B^{-1}}}(y)\chi_{_C}(\tau^{-1})
\Delta_S^{-1}(\sigma)\Delta_G^{-1}(q(\tau))\Delta_G^{-1}(y) \\
& & \hspace{9.3cm}d\lambda_S^{s}(\sigma)  d\gamma_p^{d(y)}(s) d\lambda_G^{q(d(\tau))}(y)d\mu_T(\tau)
\end{eqnarray*}
Using the quasi-invariance of $\mu_T^{(0)}$ and formula (\ref{eq:Delta inverse}) of Remark \ref{rem:Delta} gives
\begin{eqnarray*}
&=& \int_{T}\int_G\int_{S^{(0)}}\int_{S}\chi_{_A}(\sigma)\chi_{_{B^{-1}}}(y)\chi_{_C}(\tau)
\Delta_S^{-1}(\sigma)\Delta_G^{-1}(q(\tau)^{-1})\Delta_G^{-1}(y)\Delta_T^{-1}(\tau) \\
& & \hspace{9cm} d\lambda_S^{s}(\sigma) d\gamma_p^{d(y)}(s) d\lambda_G^{q(r(\tau))}(y)d\mu_T(\tau)
\end{eqnarray*}
Re-expanding $d\mu_T$ we get:
\begin{eqnarray*}
&=& \int_{G^{(0)}}\int_{T^{(0)}}\int_{T}\int_G\int_{S^{(0)}}\int_{S}\chi_{_A}(\sigma)\chi_{_{B^{-1}}}(y)\chi_{_C}(\tau)
\Delta_S^{-1}(\sigma)\Delta_G^{-1}(q(\tau)^{-1})\Delta_G^{-1}(y)\Delta_T^{-1}(\tau) \\
& & \hspace{6.5cm} d\lambda_S^{s}(\sigma) d\gamma_p^{d(y)}(s) d\lambda_G^{q(r(\tau))}(y)d\lambda_T^{t}(\tau)d\gamma_q^{u}(t)d\mu_G^{(0)}(u)
\end{eqnarray*}
We invoke Lemma \ref{lem:commuting triple integrals} once again, with $T,t,\tau$ and $q$ in place of $S,s,\sigma$ and $p$. We obtain
\begin{eqnarray*}
&=& \int_{G^{(0)}}\int_G\int_{T^{(0)}}\int_{T}\int_{S^{(0)}}\int_{S}\chi_{_A}(\sigma)\chi_{_{B^{-1}}}(y)\chi_{_C}(\tau)
\Delta_S^{-1}(\sigma)\Delta_G^{-1}(q(\tau)^{-1})\Delta_G^{-1}(y)\Delta_T^{-1}(\tau) \\
& & \hspace{6cm} d\lambda_S^{s}(\sigma) d\gamma_p^{d(y)}(s) d\lambda_T^{t}(\tau)d\gamma_q^{r(y)}(t)d\lambda^{u}_G(y)d\mu_G^{(0)}(u)
\end{eqnarray*}
By Lemma \ref{lem:integrating muG vs muG0} this equals
\begin{eqnarray*}
&=& \int_G\int_{T^{(0)}}\int_{T}\int_{S^{(0)}}\int_{S}\chi_{_A}(\sigma)\chi_{_{B^{-1}}}(y)\chi_{_C}(\tau)
\Delta_S^{-1}(\sigma)\Delta_G^{-1}(q(\tau)^{-1})\Delta_G^{-1}(y)\Delta_T^{-1}(\tau) \\
& & \hspace{7cm} d\lambda_S^{s}(\sigma)d\gamma_p^{d(y)}(s) d\lambda_T^{t}(\tau)d\gamma_q^{r(y)}(t)d\mu_G(y)
\end{eqnarray*}
We once again now use the quasi-invariance of $\mu_G^{(0)}$ and formula (\ref{eq:Delta inverse}) of Remark \ref{rem:Delta} to write
\begin{eqnarray*}
&=& \int_G\int_{T^{(0)}}\int_{T}\int_{S^{(0)}}\int_{S}\chi_{_A}(\sigma)\chi_{_{B}}(y)\chi_{_C}(\tau)
\Delta_S^{-1}(\sigma)\Delta_G^{-1}(q(\tau)^{-1})\Delta_G^{-1}(y^{-1})\Delta_G^{-1}(y)\Delta_T^{-1}(\tau) \\
& & \hspace{8cm} d\lambda_S^{s}(\sigma)d\gamma_p^{r(y)}(s) d\lambda_T^{t}(\tau)d\gamma_q^{d(y)}(t)d\mu_G(y)
\end{eqnarray*}
Returning to $\chi_{_E}$ and using Lemma \ref{lem:integrating muG vs muG0}, we get
\begin{eqnarray*}
&=& \int_{G^{(0)}}\int_G\int_{T^{(0)}}\int_{T}\int_{S^{(0)}}\int_{S}\chi_{_E}(\sigma, y, \tau)
\Delta_S^{-1}(\sigma)\Delta_G^{-1}(q(\tau)^{-1})\Delta_T^{-1}(\tau) d\lambda_S^{s}(\sigma)d\gamma_p^{r(y)}(s) \\
& & \hspace{9cm}  d\lambda_T^{t}(\tau)d\gamma_q^{d(y)}(t)d\lambda^{u}_G(y)d\mu_G^{(0)}(u) \\
\end{eqnarray*}
As we argued earlier in this calculation, we can change the order of integration:
\begin{eqnarray*}
&=& \int_{G^{(0)}}\int_G\int_{S^{(0)}}\int_{S}\int_{T^{(0)}}\int_{T}\chi_{_E}(\sigma, y, \tau)
\Delta_S^{-1}(\sigma)\Delta_G^{-1}(q(\tau)^{-1})\Delta_T^{-1}(\tau) d\lambda_T^{t}(\tau)d\gamma_q^{d(y)}(t) \\
& & \hspace{9cm} d\lambda_S^{s}(\sigma)d\gamma_p^{r(y)}(s) d\lambda^{u}_G(y)d\mu_G^{(0)}(u)
\end{eqnarray*}
Finally, we define $\Lambda(\sigma, y, \tau) = \Delta_S^{-1}(\sigma)\Delta_G^{-1}(q(\tau)^{-1})\Delta_T^{-1}(\tau)$. We get:
\begin{eqnarray*}
\mu_P^{-1}(E) &=& \int_{G^{(0)}}\int_G\int_{S^{(0)}}\int_{S}\int_{T^{(0)}}\int_{T}\chi_{_E}(\sigma, y, \tau)
\Lambda(\sigma, y, \tau) d\lambda_T^{t}(\tau)d\gamma_q^{d(y)}(t) \\
& & \hspace{8cm} d\lambda_S^{s}(\sigma)d\gamma_p^{r(y)}(s) d\lambda^{u}_G(y)d\mu_G^{(0)}(u)
\end{eqnarray*}
By Lemma \ref{lem:expanding integral muP} this equals $\displaystyle \int_P \chi_{_E}(\sigma, x, \tau) \Lambda(\sigma, x, \tau) d\mu_P(\sigma, x, \tau)$, proving the claim for any elementary open set. In order to complete the proof, we need to show that the claim holds for any Borel set $E \subseteq P$. For this, we will
invoke Lemma 2.24 of \cite{BSM}, as in the proof of Proposition \ref{prop:left invariance of lambda P}. For any Borel subset $E$, we define $$\mu(E) = \mu_P^{-1}(E) \qquad \text{ and } \qquad \nu(E)=\int_P \chi_{_E}(\alpha) \Lambda(\alpha) d\mu_P(\alpha).$$ As in Lemma \ref{lem:properties of induced measure}, since $\mu_P^{(0)}$ is locally finite and $\lambda_P^{\bullet}$ is a continuous Haar system, the induced measure $\mu_P$ is locally finite, hence so is the measure $\mu$. Thus $\nu$ is locally finite as well, since $\mu(E)=\nu(E)$ for any elementary open set $E$, and these sets constitute a basis $\mathcal{B}$ for the topology of $P$ by Remark \ref{rem:elementary open subsets of P}. Finally, $\mu$ and $\nu$ agree on finite intersections of sets in $\mathcal{B}$ as these are themselves elementary open sets, so Lemma 2.24 of \cite{BSM} implies that $\mu(E)=\nu(E)$ for all Borel sets. The proof is complete.
\end{proof}

\begin{rem}\label{rem:DeltaP}
In particular, it follows from the above calculation that the modular function of $\mu_P$ is given by
$\Delta_P (\sigma, x, \tau)= \Delta_S(\sigma) \Delta_T(\tau)/\Delta_G(q(\tau)).$
\end{rem}

%***************************************************************
%***************************************************************
\section{The weak pullback of Haar groupoids}\label{sec:WPB for HG}
%***************************************************************
%***************************************************************

We return to the weak pullback diagram, which we have now completed:
\[
\begin{array}{cccc}
\xymatrix{& & P \ar[dl]_{\pi_S} \ar[dr]^{\pi_T} & \!\!\!\!\!\!\!\!\!\!\!\!\! \scriptstyle{\lambda_P^{\bullet}, \ \mu_P^{(0)}} \\ \scriptstyle{\lambda_S^{\bullet}, \ \mu_S^{(0)}} \!\!\!\!\!\!\!\!\!\!\!\!\!\!\! & S \ar[dr]_p
& & T \ar[dl]^q
& \!\!\!\!\!\!\!\!\!\!\!\! \scriptstyle{\lambda_T^{\bullet}, \ \mu_T^{(0)}} \\ & & G & \!\!\!\!\!\!\!\!\!\!\!\!\! \scriptstyle{\lambda_G^{\bullet}, \ \mu_G^{(0)}} \\}
\end{array}
\]
In order for $(P,\lambda_P^{\bullet},\mu_P^{(0)})$ to indeed be the weak pullback in the category $\mathcal{HG}$, it must be a Haar groupoid in the sense of Definition \ref{def:Haar groupoid}, and the maps $\pi_S:P \rightarrow S$ and $\pi_T:P \rightarrow T$ need to be homomorphisms of Haar groupoids in the sense of Definition \ref{def:Haar groupoid homomorphism}. The first fact is an immediate corollary of Theorem \ref{thm:Haar system for P} and Proposition \ref{prop:muP is quasi invariant}. The second fact is proved below.

\begin{cor}
The groupoid $(P,\lambda_P^{\bullet},\mu_P^{(0)})$ is a Haar groupoid.
\end{cor}

\begin{prop}\label{prop:weak pullback maps are measure class preserving}
The maps $\pi_S:P \rightarrow S$ and $\pi_T:P \rightarrow T$ are homomorphisms of Haar groupoids.
\end{prop}

\begin{proof}
By lemma \ref{prop:piS and piT continuous homomorphisms}, the maps $\pi_S$ and $\pi_T$ are continuous groupoid homomorphisms. It remains to show that they are measure class preserving with respect to the induced measures. We prove first that $(\pi_S)_*(\mu_P) \sim \mu_S$.

Let $\Sigma \subseteq S$ be a Borel subset. Using the definition of $\mu_P$, we have
$$(\pi_S)_* (\mu_P) (\Sigma) = \mu_P(\pi_S^{-1}(\Sigma)) = \int_{P^{(0)}} \lambda_P^{(s,g,t)} (\pi_S^{-1}(\Sigma)) d \mu_P^{(0)}(s,g,t).$$
Observe that $\pi_S^{-1}(\Sigma) = \{ (\sigma, x, \tau) \in P ~|~ \sigma \in \Sigma  \} = \left( \Sigma \times G \times T \right) \cap P$. Substituting
$\lambda_P^{(s,g,t)} = \lambda_S^{s} \times \delta_g \times \lambda_T^{t}$ according to Definition \ref{def:Haar for P}, and noting that systems of measures are concentrated on fibers, we get:
\begin{eqnarray*}
\lambda_P^{(s,g,t)} (\pi_S^{-1}(\Sigma)) &=& \lambda_P^{(s,g,t)} \left((\Sigma \times G \times T) \cap P^{(s,g,t)} \right) \\
&=& (\lambda_S^{s} \times \delta_g \times \lambda_T^{t}) \left( (\Sigma \times G \times T) \cap ( S^s \times \{g\} \times T^t) \right) \qquad (\text{by Lemma \ref{lem:fibers of P}}) \\
&=& \lambda_S^{s} (\Sigma) \cdot \delta_g (\{g\}) \cdot \lambda_T^{t} (T) \\
&=& \lambda_S^{s} (\Sigma)  \lambda_T^{t} (T)
\end{eqnarray*}
Therefore, using Lemma \ref{lem:expanding integral muP} and then rewriting $\eta^y$ by Proposition \ref{prop:eta}, we have
\begin{eqnarray*}
\hspace{-2cm}(\pi_S)_* (\mu_P) (\Sigma) &=& \int_G \int_{P^{(0)}} \lambda_S^{s}(\Sigma) \lambda_T^{t}(T) d\eta^y (s,g,t) d\mu_G (y) \\
&=& \int_G\int\!\!\!\!\int\!\!\!\!\int_{S^{(0)}\times G\times T^{(0)}}\lambda_S^{s}(\Sigma) \lambda_T^{t}(T)d\gamma_p^{r(y)}(s)d\delta_y(g)d\gamma_q^{d(y)}(t)d\mu_G(y) \\
&=& \int_G\int\!\!\!\!\int_{S^{(0)}\times T^{(0)}}\lambda_S^{s}(\Sigma) \lambda_T^{t}(T)d\gamma_p^{r(y)}(s)d\gamma_q^{d(y)}(t)d\mu_G(y)
\end{eqnarray*}
We use Fubini's theorem, as well as Lemma \ref{lem:integrating muG vs muG0}, to obtain
\begin{eqnarray*}
\hspace{0.8cm} &=& \int_{G^{(0)}}\int_G\int_{S^{(0)}}\int_{T^{(0)}}\lambda_S^{s}(\Sigma) \lambda_T^{t}(T)d\gamma_q^{d(y)}(t)d\gamma_p^{r(y)}(s)d\lambda_G^u (y)d\mu_G^{(0)}(u)
\end{eqnarray*}
Furthermore, the fact that $\lambda_G^u$ is supported on $G^u$ dictates that $r(y)=u$, hence we get
\begin{eqnarray*}
\hspace{0.5cm} &=& \int_{G^{(0)}}\int_G\int_{S^{(0)}}\lambda_S^{s}(\Sigma) \int_{T^{(0)}} \lambda_T^{t}(T)d\gamma_q^{d(y)}(t)d\gamma_p^{u}(s)d\lambda_G^u (y)d\mu_G^{(0)}(u) \\
&=& \int_{G^{(0)}}\int_{S^{(0)}}\lambda_S^{s}(\Sigma)\int_G \int_{T^{(0)}} \lambda_T^{t}(T)d\gamma_q^{d(y)}(t)d\lambda_G^u (y)d\gamma_p^{u}(s)d\mu_G^{(0)}(u)
\end{eqnarray*}
We now define a function $h_1$ on $G^{(0)}$ by $$h_1(u) = \int_G \int_{T^{(0)}} \lambda_T^{t}(T)d\gamma_q^{d(y)}(t)d\lambda_G^u (y).$$ Since $\lambda_T^{t}(T) >0$ for any $t$, the function $h_1(u)$ is strictly positive on $G^{(0)}$. Returning to our main calculation, we have:
\begin{eqnarray*}
(\pi_S)_* (\mu_P) (\Sigma) &=& \int_{G^{(0)}}\int_{S^{(0)}}\lambda_S^{s}(\Sigma) h_1(u) d\gamma_p^{u}(s)d\mu_G^{(0)}(u) \\
&=& \int_{G^{(0)}}\int_{S^{(0)}}\lambda_S^{s}(\Sigma) h_1(p(s)) d\gamma_p^{u}(s)d\mu_G^{(0)}(u)
\end{eqnarray*}
since $\gamma_p^{u}$ is concentrated on $p^{-1}(u)$. Finally, $\gamma_p^{\bullet}$ is a disintegration of $\mu_S^{(0)}$ with respect to $\mu_G^{(0)}$, hence
\begin{eqnarray*}
(\pi_S)_* (\mu_P) (\Sigma) &=& \int_{S^{(0)}}\lambda_S^{s}(\Sigma) h_1(p(s)) d\mu_S^{(0)}(s)
\end{eqnarray*}
On the other hand, $\displaystyle \mu_S(\Sigma) = \int_{S^{(0)}}\lambda_S^{s}(\Sigma) d\mu_S^{(0)}(s).$ It follows that $\mu_S(\Sigma) = 0$ if and only if $(\pi_S)_* (\mu_P) (\Sigma) =0$. \\

We turn to $\pi_T$. Proving that $(\pi_T)_*(\mu_P) \sim \mu_T$ will require a detour via the quasi-invariance of $\mu_G^{(0)}$. Let $\Omega \subseteq T$ be a Borel subset. Tracing the line of arguments above, we have $$(\pi_T)_* (\mu_P) (\Omega) = \mu_P(\pi_T^{-1}(\Omega)) = \int_{P^{(0)}} \lambda_P^{(s,g,t)} (\pi_T^{-1}(\Omega)) d \mu_P^{(0)}(s,g,t),$$ where
\begin{eqnarray*}
\lambda_P^{(s,g,t)} (\pi_T^{-1}(\Omega)) &=& \lambda_P^{(s,g,t)} \left((S \times G \times \Omega) \cap P^{(s,g,t)} \right) \\
&=& (\lambda_S^{s} \times \delta_g \times \lambda_T^{t}) \left( (S \times G \times \Omega) \cap ( S^s \times \{g\} \times T^t) \right) \\
&=& \lambda_S^{s} (S)  \lambda_T^{t} (\Omega)
\end{eqnarray*}
Therefore,
\begin{eqnarray*}
\hspace{-1.8cm} (\pi_T)_* (\mu_P) (\Omega) &=& \int_G \int_{P^{(0)}} \lambda_S^{s} (S)  \lambda_T^{t} (\Omega) d\eta^y (s,g,t) d\mu_G (y) \\
&=& \int_G\int\!\!\!\!\int\!\!\!\!\int_{S^{(0)}\times G\times T^{(0)}}\lambda_S^{s}(S) \lambda_T^{t}(\Omega)d\gamma_p^{r(y)}(s)d\delta_y(g)d\gamma_q^{d(y)}(t)d\mu_G(y) \\
&=& \int_G\int_{T^{(0)}}\int_{S^{(0)}}\lambda_S^{s}(S)\lambda_T^{t} (\Omega)d\gamma_p^{r(y)}(s)d\gamma_q^{d(y)}(t)d\mu_G(y)
\end{eqnarray*}
Using the quasi-invariance of $\mu_G^{(0)}$ and formula (\ref{eq:Delta inverse}) of Remark \ref{rem:Delta}, we get
\begin{eqnarray*}
&=& \int_G\int_{T^{(0)}}\int_{S^{(0)}}\lambda_S^{s}(S)\lambda_T^{t} (\Omega)\Delta_G^{-1}(y)d\gamma_p^{d(y)}(s)d\gamma_q^{r(y)}(t)d\mu_G(y)
\end{eqnarray*}
Replacing $\gamma_q^{r(y)}$ by $\gamma_q^u$ as before, and using Lemma \ref{lem:integrating muG vs muG0} and Fubini's theorem, we get
\begin{eqnarray*}
&=& \int_{G^{(0)}}\int_G\int_{T^{(0)}}\lambda_T^{t} (\Omega)\int_{S^{(0)}}\lambda_S^{s}(S)\Delta_G^{-1}(y)d\gamma_p^{d(y)}(s)d\gamma_q^{u}(t)d\lambda_G^u (y)d\mu_G^{(0)}(u) \\
&=& \int_{G^{(0)}}\int_{T^{(0)}}\lambda_T^{t} (\Omega)\int_G\int_{S^{(0)}}\lambda_S^{s}(S)\Delta_G^{-1}(y)d\gamma_p^{d(y)}(s)d\lambda_G^u (y)d\gamma_q^{u}(t)d\mu_G^{(0)}(u)
\end{eqnarray*}
The function $h_2$ on $G^{(0)}$ defined by $$h_2(u) = \int_G \int_{S^{(0)}} \lambda_S^{s}(S)\Delta_G^{-1}(y)d\gamma_p^{d(y)}(s)d\lambda_G^u (y)$$ is positive since $\lambda_S^{s}(S) >0$ for any $s$ and the modular function $\Delta_G$ is positive. Returning to our main calculation, we have:
\begin{eqnarray*}
(\pi_T)_* (\mu_P) (\Omega) &=& \int_{G^{(0)}}\int_{T^{(0)}}\lambda_T^{t} (\Omega) h_2(u) d\gamma_q^{u}(t)d\mu_G^{(0)}(u) \\
&=& \int_{G^{(0)}}\int_{T^{(0)}}\lambda_T^{t} (\Omega) h_2(q(t)) d\gamma_q^{u}(t)d\mu_G^{(0)}(u) \\
\end{eqnarray*}
since $\gamma_q^{u}$ is concentrated on $q^{-1}(u)$. Finally, $\gamma_q^{\bullet}$ is a disintegration of $\mu_T^{(0)}$ with respect to $\mu_G^{(0)}$, hence
\begin{eqnarray*}
(\pi_T)_* (\mu_P) (\Omega) &=& \int_{T^{(0)}}\lambda_T^{t}(\Omega) h_2(q(t)) d\mu_T^{(0)}(t)
\end{eqnarray*}
On the other hand, $\displaystyle \mu_T(\Omega) = \int_{T^{(0)}}\lambda_T^{t}(\Omega)d\mu_T^{(0)}(t).$ It follows that $\mu_T(\Omega) = 0$ if and only if $(\pi_T)_* (\mu_P) (\Omega) =0$. This completes the proof.
\end{proof}

Recall our standing Assumption \ref{ass:LB disintegration}, by which the maps $p$ and $q$ (restricted to the unit spaces) admit disintegrations which are locally bounded. As we show in the following proposition, the map $\pi_S$ will automatically inherit this property. However, in order to guarantee that the map $\pi_T$ admits a disintegration which is locally bounded, we will need another assumption.

\begin{ass}\label{ass:Delta_G locally bounded}
We will assume that the modular function $\Delta_G$ is locally bounded on $G$, in the sense that for every point $x \in G$ there exist a neighborhood $U_x$ and positive constants $c_x$ and $C_x$ such that $c_x < \Delta_G (y) < C_x$ for every $y \in U_x$.
\end{ass}

Note that $\Delta_G^{-1}$ is locally bounded whenever $\Delta_G$ is locally bounded.

\begin{rem}
If we assume that $\Delta_S$ and $\Delta_T$ are also locally bounded in the above sense, then Remark \ref{rem:DeltaP} implies that $\Delta_P$ is locally bounded as well.
\end{rem}

\begin{prop}\label{prop:weak pullback maps admit LB disintegrations}
The maps $\pi_S:P^{(0)}\rightarrow S^{(0)}$ and $\pi_T:P^{(0)}\rightarrow T^{(0)}$ admit disintegrations which are locally bounded.
\end{prop}

\begin{proof}
We start with the map $\pi_S$. We shall use Proposition 6.8 from \cite{BSM}, which provides a necessary and sufficient condition for admitting a disintegration which is locally bounded: for any compact set $K\subseteq P^{(0)}$ there must exist a constant $C_{_K}$ such that for all Borel sets $\Sigma \subseteq S^{(0)}$, $\mu_P^{(0)}(K\cap \pi_S^{-1}(\Sigma))\leq C_{_K}\cdot \mu_S^{(0)}(\Sigma).$

Let $K\subseteq P^{(0)}$ be compact. Consider three increasing sequences $\{A_n\}$, $\{B_n\}$ and $\{C_n\}$ of open subsets with compact closures in $S$, $G$ and $T$ respectively, such that $S=\bigcup_{n=1}^{\infty}A_n$, $G=\bigcup_{n=1}^{\infty}B_n$, and $T=\bigcup_{n=1}^{\infty}C_n$ (such sequences exist in any locally compact second countable space). The elementary open sets $E_n=(A_n\times B_n\times C_n)\cap P^{(0)}$ determine an increasing open cover of $P^{(0)}$ and in particular of $K$. Since $K$ is compact, $K\subseteq E_i$ for some $i$. Denoting $K_1 = \overline{A_i}$, $K_2 = \overline{B_i}$ and $K_3 = \overline{C_i}$, we have $K\subseteq (K_1\times K_2\times K_3)\cap P^{(0)}$ where $K_1 \subseteq S$, $K_2 \subseteq G$ and $K_3 \subseteq T$ are each compact.

For any Borel set $\Sigma \subseteq S^{(0)}$,
\begin{eqnarray*}
\mu_P^{(0)}(K\cap \pi_S^{-1}(\Sigma)) &\leq& \mu_P^{(0)}((K_1\!\times\! K_2\!\times\! K_3)\cap \pi_S^{-1}(\Sigma)) \ = \ \mu_P^{(0)}(((K_1\cap\Sigma)\times K_2\times K_3)\cap P^{(0)}) \\ &=& \int_{K_2}\gamma_p^{r(x)}(K_1\cap\Sigma)\gamma_q^{d(x)}(K_3) d\mu_G(x)
\end{eqnarray*}
where the last equality follows from a calculation as in the proof of Proposition \ref{prop:muP0 is Radon}. Expanding $\mu_G$ we get
\begin{eqnarray*}
&=& \int_{G^{(0)}}\int_{K_2}\gamma_p^{r(x)}(K_1\cap\Sigma)\gamma_q^{d(x)}(K_3) d\lambda_G^u(x)d\mu_G^{(0)}(u) \\
&\leq& \int_{G^{(0)}}\int_{K_2}\gamma_p^{r(x)}(\Sigma)\gamma_q^{d(x)}(K_3) d\lambda_G^u(x)d\mu_G^{(0)}(u)
\end{eqnarray*}
Next, we note that $r(x) = u$ since $\lambda_G^u$ is supported on $r^{-1}(u)$, and then rewrite $\gamma_p^{u}(\Sigma)$:
\begin{eqnarray*}
&=& \int_{G^{(0)}}\int_{K_2}\gamma_p^{u}(\Sigma)\gamma_q^{d(x)}(K_3)d\lambda_G^u(x)d\mu_G^{(0)}(u) \\
&=& \int_{G^{(0)}}\int_{K_2}\int_{S^{(0)}}\chi_{_\Sigma}(s)\gamma_q^{d(x)}(K_3) d\gamma_p^{u}(s) d\lambda_G^u(x)d\mu_G^{(0)}(u)
\end{eqnarray*}
We use Fubini's Theorem and note that $p(s)=u$ since $\gamma_p^{u}$ is supported on $p^{-1}(u)$, after which we can collapse the outer two integrals, since $\gamma_p$ is a disintegration:
\begin{eqnarray*}
&=& \int_{G^{(0)}}\int_{S^{(0)}} \int_{K_2} \chi_{_\Sigma}(s)\gamma_q^{d(x)}(K_3)  d\lambda_G^{p(s)}(x) d\gamma_p^{u}(s) d\mu_G^{(0)}(u) \\
&=& \int_{S^{(0)}}\int_{K_2}\chi_{_\Sigma}(s)\gamma_q^{d(x)}(K_3) d\lambda_G^{p(s)}(x)d\mu_S^{(0)}(s) \ \leq \ C\cdot\mu_S^{(0)}(\Sigma),
\end{eqnarray*}
where $\displaystyle C=\left(\sup_u \gamma_q^{u}(K_3)\right)\cdot\left(\sup_v\lambda_G^{v}(K_2)\right)$. Both suprema exist since $\gamma_q^{\bullet}$ and $\lambda_G^{\bullet}$ are locally bounded, hence bounded on compact sets.

We turn to the map $\pi_T$. The proof will be analogous, but will require the use of the function $\Delta_G^{-1}$, which is locally bounded by Assumption \ref{ass:Delta_G locally bounded}. Let $\Omega\subseteq T^{(0)}$.
\begin{eqnarray*}
\mu_P^{(0)}(K\cap \pi_T^{-1}(\Omega)) &\leq& \mu_P^{(0)}((K_1\!\times\! K_2\!\times\! K_3)\cap \pi_T^{-1}(\Omega)) \ = \ \mu_P^{(0)}((K_1\times K_2\times (K_3\cap\Omega))\cap P^{(0)}) \\
&=& \int_{K_2}\gamma_p^{r(x)}(K_1)\gamma_q^{d(x)}(K_3\cap\Omega) d\mu_G(x) \\
&=& \int_{K_2^{-1}}\gamma_p^{d(x)}(K_1)\gamma_q^{r(x)}(K_3\cap\Omega)\Delta_G^{-1}(x)\ d\mu_G(x)
\end{eqnarray*}
Skipping intermediate calculations which mimic the $\pi_S$ case, we get
\begin{eqnarray*}
&\leq& \int_{T^{(0)}}\int_{K_2^{-1}}\chi_{\Omega}(t)\gamma_p^{d(x)}(K_1)\Delta_G^{-1}(x) d\lambda_G^{q(t)}(x)d\mu_T^{(0)}(t) \ \leq \ D \cdot \mu_T^{(0)}(\Omega)
\end{eqnarray*}
where $\displaystyle D=\left(\sup_u\gamma_p^{u}(K_1)\right)\cdot \left(\sup_{x\in K_2^{-1}} \Delta_G^{-1}(x) \right)\cdot \left(\sup_v\lambda_G^{v}(K_2^{-1})\right)$. All suprema exist since $\gamma_p^{\bullet}$ and $\lambda_G^{\bullet}$ are bounded on compact sets, and $\Delta_G^{-1}$ is locally bounded.
\end{proof}

%***************************************************************
%***************************************************************
\section*{Acknowledgments}
%***************************************************************
%***************************************************************

We thank John Baez, Christopher Walker and most of all Paul Muhly for inspiring discussions and useful remarks.


\begin{thebibliography}{100}

\bibitem{renault-anantharaman-delaroche}
Claire Anantharaman-Delaroche and Jean Renault, \emph{Amenable groupoids}, Monographies de L'En\-seigne\-ment Math\'ema\-tique, volume 36, Geneva, 2000.

\bibitem{baez-hoffnung-walker}
John~C. Baez, Alexander~E. Hoffnung and Christopher~D. Walker, \emph{Higher-dimensional algebra VII: groupoidification}, 2009, preprint arXiv:0908.4305v2  [math.QA].

\bibitem{blanchard}
\'Etienne Blanchard, \emph{D\'eformations de {$C\sp{\ast} $}-alg\`ebres de Hopf}, Bull. Soc. Math. France 124 (1996), pp. 141--215.

\bibitem{BSM}
Aviv Censor and Daniele Grandini, \emph{Borel and continuous systems of measures}, 2010, preprint arXiv:1004.3750v1 [math.FA].

\bibitem{Hahn1}
Peter Hahn, \emph{Haar measure for measure groupoids}, Trans. Amer. Math. Soc. vol. 242 (1978), pp. 1--33.

\bibitem{Mackey}
George~W. Mackey, \emph{Ergodic theory, group theory, and differential geometry}, Proc. Nat. Acad. Sci. U.S.A.
50 (1963), pp. 1184--1191

\bibitem{muhly-book-unpublished}
Paul~S. Muhly, \emph{Coordinates in operator algebras}, to appear in CBMS lecture notes series.

\bibitem{paterson-book}
Alan~L.~T. Paterson, \emph{Groupoids, inverse semigroups, and their operator algebras}, Progress in Mathematics, volume 170, Birkhauser, Boston, 1999.

\bibitem{Ramsay polish groupoids}
Arlan~B. Ramsay, \emph{Polish groupoids}, in \emph{Descriptive set theory and dynamical systems}, London Math. Soc. Lecture Note Series, volume 277, pp. 259--271, Cambridge University Press, Cambridge, 2000.

\bibitem{Ramsay virtual groups}
Arlan~B. Ramsay, \emph{Virtual groups and group actions}, Advances in Math. 6 (1971), pp. 253--322.

\bibitem{renault-book}
Jean Renault, \emph{A groupoid approach to {$C\sp{\ast}$}-algebras}, Lecture Notes in Mathematics, volume 793, Springer, Berlin, 1980.

\bibitem{Seda80}
Anthony~K. Seda, \emph{On measures in fibre spaces}, Cahiers de Topologie est G\'eom\'etrie Differentielle Cat\'egoriques, vol.21 no.3 (1980) pp. 247--276.

\end{thebibliography}
\end{document}